\pdfoutput=1
\RequirePackage{ifpdf}
\ifpdf % We are running pdfTeX in pdf mode
\documentclass[pdftex]{sigma}
\else
\documentclass{sigma}
\fi

\numberwithin{equation}{section}

\newtheorem*{MainTheorem}{Main Theorem}
\newtheorem{Theorem}{Theorem}[section]
\newtheorem{Corollary}[Theorem]{Corollary}
\newtheorem{Lemma}[Theorem]{Lemma}
\newtheorem{Proposition}[Theorem]{Proposition}
 { \theoremstyle{definition}
\newtheorem{Definition}[Theorem]{Definition}
\newtheorem{Remark}[Theorem]{Remark} }

\begin{document}

\allowdisplaybreaks

\newcommand{\arXivNumber}{1809.02951}

\renewcommand{\PaperNumber}{038}

\FirstPageHeading

\ShortArticleName{The Laurent Extension of Quantum Plane: a Complete List of $U_q(\mathfrak{sl}_2)$-Symmetries}

\ArticleName{The Laurent Extension of Quantum Plane:\\ a Complete List of $\boldsymbol{U_q(\mathfrak{sl}_2)}$-Symmetries}

\Author{Sergey SINEL'SHCHIKOV}

\AuthorNameForHeading{S.~Sinel'shchikov}

\Address{Mathematics Division, B.~Verkin Institute for Low Temperature Physics and Engineering\\ of the National Academy of Sciences of Ukraine, 47 Nauky Ave., 61103 Kharkiv, Ukraine}
\Email{\href{mailto:sinelshchikov@ilt.kharkov.ua}{sinelshchikov@ilt.kharkov.ua}}
\URLaddress{\url{http://ilt.kharkov.ua/bvi/structure/depart_e/d25/d25e.htm}}

\ArticleDates{Received September 11, 2018, in final form April 17, 2019; Published online May 09, 2019}

\Abstract{This work finishes a classification of $U_q(\mathfrak{sl}_2)$-symmetries on the Laurent extension $\mathbb{C}_q\big[x^{\pm 1},y^{\pm 1}\big]$ of the quantum plane. After reproducing the partial results of a previous paper of the author related to symmetries with non-trivial action of the Cartan generator(s) of $U_q(\mathfrak{sl}_2)$ and the generic symmetries, a complete collection of non-generic symmetries is presented. Together, these collections constitute a complete list of $U_q(\mathfrak{sl}_2)$-symmetries on $\mathbb{C}_q\big[x^{\pm 1},y^{\pm 1}\big]$.}


\Classification{81R50; 17B37}

\section{Introduction}

Quantum algebras are normally considered together with a distinguished collection (`quantum group') of their symmetries that form a sort of quantum dynamical system (a precise definition is given below). A good example in this context is the quantum plane $\mathbb{C}_q[x,y]$, which is among the simplest quantum algebras. It is very well known that $\mathbb{C}_q[x,y]$ carries a $U_q(\mathfrak{sl}_2)$-symmetry (in other terms, a structure of $U_q(\mathfrak{sl}_2)$-module algebra, see, e.g.,~\cite{C}). During a long time it was the only such symmetry considered in the literature.

The initial approach to considering different symmetries has been developed in the paper by S.~Duplij and S.~Sinel'shchikov~\cite{DS}. The authors produced a complete list of $U_q(\mathfrak{sl}_2)$-symmetries on $\mathbb{C}_q[x,y]$ and, in particular, demonstrated the existence of an uncountable collection of pairwise non-isomorphic such symmetries.

The next and very natural step has been done in a work by S.~Duplij, Y.~Hong, and F.~Li~\cite{DHL}, where the structures of $U_q(\mathfrak{sl}_m)$-module algebra on a generalized quantum plane, a polynomial algebra in $n$
quasi-commuting variables, $m,n>2$, are considered. The authors apply, in the above broader context, the methods of~\cite{DS}, together with more enhanced tools related to (ge\-ne\-ralized) Dynkin diagrams.

Another approach to generalizing the results of \cite{DS} and different from increasing the dimension parameters as in \cite{DHL}, has been developed in \cite{S}. The basic motive was in observing that a complete list of symmetries on $\mathbb{C}_q[x,y]$ (as well as in \cite{DHL} within fixed dimension(s)) is formed by only finitely many series labelled by pairs of weight constants; the latter constitute invariants of isomorphism of symmetries. This hints that the resulting amount of symmetries is rather low, and one may try to embed everything into a more symmetric algebra, where restricting to the initial subalgebra $\mathbb{C}_q[x,y]$ breaks the symmetry very essentially. Fortunately, the required embedding is readily at a hand: the embedding to the algebra $\mathbb{C}_q\big[x^{\pm 1},y^{\pm 1}\big]$ of Laurent polynomials on the quantum plane. This, together with a bit surprising (but in fact completely straightforward) observation that every $U_q(\mathfrak{sl}_2)$-symmetry of the subalgebra $\mathbb{C}_q[x,y]$ admits an extension to a symmetry of the larger algebra $\mathbb{C}_q\big[x^{\pm 1},y^{\pm 1}\big]$, allows one to claim that, in view of the results presented below, $\mathbb{C}_q\big[x^{\pm 1},y^{\pm 1}\big]$ is much more symmetric than its subalgebra $\mathbb{C}_q[x,y]$. Thus, studying the symmetries of $\mathbb{C}_q\big[x^{\pm 1},y^{\pm 1}\big]$ appears to be significant because it demonstrates where the symmetries of the standard quantum plane $\mathbb{C}_q[x,y]$ come from. The extended algebra, while retaining the polynomial nature (all the sums are finite), removes the effect of presence of the `lowest homogeneous component', the principal obstacle for existence of a large collection of symmetries.

The purpose of this work is to extend the list of $U_q(\mathfrak{sl}_2)$-symmetries on $\mathbb{C}_q\big[x^{\pm 1},y^{\pm 1}\big]$ presented in \cite{S} up to a complete list of symmetries. One might expect that a huge amount of additional symmetries as compared to those coming from the subalgebra $\mathbb{C}_q[x,y]$, should be a consequence of a larger group of automorphisms for $\mathbb{C}_q\big[x^{\pm 1},y^{\pm 1}\big]$ than that of $\mathbb{C}_q[x,y]$. However, it was demonstrated in \cite{S} that the collection of symmetries in which the action of Cartan generator of~$U_q(\mathfrak{sl}_2)$ (which anyway acts by an automorphism) does not reduce to multiplying the generators of quantum plane by (weight) constants, is rather poor. On the other hand, a collection of generic symmetries~\cite{S} covers infinitely (in fact, uncountably) many admissible pairs of weight constants, whence an uncountable family of non-isomorphic symmetries; this collection constitutes a single series and is disjoint from the symmetries extended from $\mathbb{C}_q[x,y]$. Another part of the complete collection, which is presented in this work, the non-generic symmetries, although covering only a countable family of weight constants, is much more extended in writing down the specific series submitted to the parameters~$D$,~$G$, and~$L$ (see Section~\ref{wpem} for their definition); this collection is presented below.

Let us describe the general classification of $U_q(\mathfrak{sl}_2)$-symmetries on $\mathbb{C}_q\big[x^{\pm 1},y^{\pm 1}\big]$, which was partially introduced in \cite{S}. The entire collection of symmetries splits into $4$ disjoint subcollections as follows:
\begin{enumerate}\itemsep=0pt
\item[I] The subcollection with the Cartan generator $\mathsf{k}$ of $U_q(\mathfrak{sl}_2)$ acting in a non-trivial way (its action does not reduce to multiplying the generators $x$, $y$ of quantum plane by constants). This subcollection is rather poor and is described by Theorem~\ref{sigma=-I}~\cite[Theorem 3.5]{S}.

\item[II] (Trivial collection with weight constants.) This subcollection is contained in the rest of symmetries which are contrary to~(I) in the sense that the Cartan generator $\mathsf{k}$ multiplies the generators $x$, $y$ of quantum plane by weight constants, and constitutes a trivial part of the complement to type~(I) symmetries. This part is formed by just $4$ symmetries in which the weight constants are $\pm 1$ and the generators $\mathsf{e}$ and $\mathsf{f}$ of $U_q(\mathfrak{sl}_2)$ act as identically zero operators. See Theorem~\ref{T01}.

\item[III] (Generic symmetries.) A part of symmetries in which the Cartan generator $\mathsf{k}$ multiplies $x$, $y$ by weight constants but, on the contrary to (II), either $\mathsf{e}$ or $\mathsf{f}$ acts as a non-zero operator. This part (subcollection) of symmetries, referred to as generic symmetries, is distinguished by a special `rational independence' assumption on the weight constants. It constitutes the most massive collection of symmetries in the sense that the associated pairs of weight constants cover all but a countable subset of possible values for such pairs. See Theorem~\ref{gener} \cite[Theorem~4.1]{S}.

\item[IV] (Non-generic symmetries.) The rest of symmetries with either $\mathsf{e}$ or $\mathsf{f}$ acting as a non-zero operator and whose weight constants break the `rational independence' assumption mentioned in (III), are referred to as non-generic symmetries and constitute the principal subjects of this work. See Main Theorem formulated below, along with the related terminology and definitions in Sections~\ref{wpem} and~\ref{clngs}.
\end{enumerate}

The above classification is not intended to describe the isomorphism classes of $U_q(\mathfrak{sl}_2)$-symmetries on $\mathbb{C}_q\big[x^{\pm 1},y^{\pm 1}\big]$; it is much coarser. Instead, it focuses on the structure of symmetries which allows one to write down suitable general formulas for the symmetries in each of types (I)--(III) as well as within the specific series in type~(IV) listed in Main Theorem. Both isomorphism and non-isomorphism statements are formulated where they are easily available. In particular, it is very well obvious to observe that any two symmetries of different types (I)--(IV) are non-isomorphic. Another point worth mentioning is Proposition \ref{ufnis} where the existence of uncountable family of non-isomorphic generic (type (III)) symmetries is established.

The following Theorem presents the list of series for non-generic (type (IV)) symmetries. The names of series in the list reflect the associated values for the invariants $D$ and $G$ mentioned above, together with the number of terms in the formulas involved. More details are to be found in Sections~\ref{wpem} and~\ref{clngs}.

\begin{MainTheorem} The collection of non-generic $U_q(\mathfrak{sl}_2)$-symmetries on $\mathbb{C}_q\big[x^{\pm 1},y^{\pm 1}\big]$ is given by the series
\begin{alignat*}{7}
&D_1G_1E_1F_3, \ \ && D_1G_1E_2F_4, \  \ && D_1G_1E_3F_3, \  \ && D_1G_1E_2F_2,\ \ && D_1G_1E_4F_2, \ \ &&D_1G_1E_3F_1,&\\
&D_2G_1E_1F_3(a),\  &&D_2G_1E_2F_2(a),\  &&D_2G_1E_3F_1(a),\  && && && &\\
&D_2G_1E_1F_3(b), &&D_2G_1E_2F_2(b), &&D_2G_1E_3F_1(b), && && && & \\
 &D_2G_2E_1F_3, &&D_2G_2E_2F_2, &&D_2G_2E_3F_1,&& && && & \\
 &D_4G_1E_1F_2(a), &&D_4G_1E_2F_1(a), &&&&&&& &&  \\
 &D_4G_1E_1F_2(b), &&D_4G_1E_2F_1(b), &&&&&&&&& \\
 &D_4G_2E_1F_2(a), \ \ &&D_4G_2E_2F_1(a), \ \ &&D_4G_2E_1F_2(b),\ \  &&D_4G_2E_2F_1(b), \ \  && &
\end{alignat*}
described in Section~{\rm \ref{clngs}}. Any two symmetries contained in the series from different lines of this table are non-isomorphic.

This collection, together with the symmetries described in Section~{\rm \ref{nw}} by Theorems~{\rm \ref{sigma=-I}},~{\rm \ref{T01}},~{\rm \ref{gener}} $($types {\rm (I)--(III))}, form a complete list of $U_q(\mathfrak{sl}_2)$-symmetries on $\mathbb{C}_q\big[x^{\pm 1},y^{\pm 1}\big]$.
\end{MainTheorem}

The outline of this paper is as follows. Section \ref{prel} (preliminaries) collects the basic definitions and facts related to our subjects. Section \ref{nw} reproduces (the formulations of) the results of \cite{S} concerning the (type (I)) symmetries with non-trivial action of the Cartan generator, type (II) (trivial) symmetries, and type~(III) (generic) symmetries. Section~\ref{wpem} introduces and studies the notion of extreme monomials for the weight polynomials related to symmetries. Certain integral parameters of symmetries are considered; among those, the parameters~$D$ and~$G$ are proved to be invariants of isomorphism of symmetries; these are to be used to label the series of non-generic symmetries. Admissible values for $D$ and $G$ are clarified, along with a~general form of a~symmetry up to complex coefficients and weight constants, to be computed later on.

Section \ref{clngs} presents the final form of the series of non-generic (type (IV)) symmetries, preceded by clarifying the associated pairs of weight constants.

\section{Preliminaries}\label{prel}

We start with recalling the general definition as follows. Let $H$ be a Hopf algebra whose comultiplication is $\Delta$, counit is $\varepsilon$, and antipode is $S$ \cite{abe}. Consider also a unital algebra $A$ whose unit is $\mathbf{1}$. The Sweedler sigma-notation related to the comultiplication $\Delta(h)=\sum\limits_{(h)}h_{(1)}\otimes h_{(2)}$ as in \cite{sweedler} is used below.

\begin{Definition}\label{symdef}
By a structure of $H$-module algebra on $A$ (to be referred to as an $H$-symmetry for the sake of brevity, or even merely a symmetry if $H$ is completely determined by the context) we mean a homomorphism of algebras $\pi\colon H\to\operatorname{End}_\mathbb{C}A$ such that
\begin{enumerate}\itemsep=0pt\samepage
\item[(i)] $\pi(h)(ab)=\sum\limits_{(h)}\pi(h_{(1)})(a)\cdot\pi(h_{(2)})(b)$ for all $h\in H$, $a,b\in A$;
\item[(ii)] $\pi(h)(\mathbf{1})=\varepsilon(h)\mathbf{1}$ for all $h\in H$.
\end{enumerate}
The structures $\pi_1$, $\pi_2$ are said to be isomorphic if there exists an automorphism $\Psi$ of the algebra~$A$ such that $\Psi\pi_1(h)\Psi^{-1}=\pi_2(h)$ for all $h\in H$.
\end{Definition}

Throughout the paper we assume that $q\in\mathbb{C}{\setminus}\{0\}$ is not a root of $1$ ($q^n\ne 1$ for all non-zero integers $n$). Consider the quantum plane which is a unital algebra $\mathbb{C}_q[x,y]$ with two genera\-tors~$x$,~$y$ and a~single relation
\begin{gather}\label{qpr}
yx=qxy.
\end{gather}

Let us complete the list of generators with two more elements $x^{-1}$, $y^{-1}$, and the list of relations with
\begin{gather}\label{Lpr}
xx^{-1}=x^{-1}x=yy^{-1}=y^{-1}y=\mathbf{1}.
\end{gather}
The extended unital algebra $\mathbb{C}_q\big[x^{\pm 1},y^{\pm 1}\big]$ defined this way is called the Laurent extension of the quantum plane (more precisely, the algebra of Laurent polynomials over the quantum plane).

Given an integral matrix $\sigma=\left(\begin{smallmatrix}k & m\\ l & n\end{smallmatrix}\right)\in {\rm SL}(2,\mathbb{Z})$ and a pair of non-zero complex numbers $(\mu,\nu)\in(\mathbb{C}{\setminus}\{0\})^2$, we associate an automorphism $\varphi_{\sigma,\mu,\nu}$ of $\mathbb{C}_q\big[x^{\pm 1},y^{\pm 1}\big]$ determined on the generators~$x$ and~$y$ by
\begin{gather}\label{Auto}
\varphi_{\sigma,\mu,\nu}(x)=\mu x^ky^m, \qquad \varphi_{\sigma,\mu,\nu}(y)=\nu x^ly^n.
\end{gather}
A well-known result claims that every automorphism of $\mathbb{C}_q\big[x^{\pm 1},y^{\pm 1}\big]$ has the form \eqref{Auto}, and the group $\operatorname{Aut}\big(\mathbb{C}_q\big[x^{\pm 1},y^{\pm 1}\big]\big)$ of automorphisms of $\mathbb{C}_q\big[x^{\pm 1},y^{\pm 1}\big]$ is just the semidirect product of its subgroups ${\rm SL}(2,\mathbb{Z})$ and $(\mathbb{C}{\setminus}\{0\})^2$ determined by setting~\cite{KPS}
\begin{gather*}
\sigma(\mu,\nu)\sigma^{-1}=(\mu,\nu)^\sigma
\stackrel{\operatorname{def}}{=}\big(\mu^k\nu^m,\mu^l\nu^n\big).
\end{gather*}
(see also \cite{AD,PLCCN}). We also rephrase this via introducing the associated action of ${\rm SL}(2,\mathbb{Z})$ by group automorphisms of $(\mathbb{C}{\setminus}\{0\})^2$
\begin{gather}\label{sl_ac}
\begin{pmatrix}k & m\\ l & n\end{pmatrix}\binom{\mu}{\nu}=\binom{\mu^k\nu^m}{\mu^l\nu^n}.
\end{gather}

The quantum universal enveloping algebra $U_q(\mathfrak{sl}_2)$ is a unital associative algebra defined by its (Chevalley) generators $\mathsf{k}$, $\mathsf{k}^{-1}$, $\mathsf{e}$, $\mathsf{f}$, and the relations
\begin{gather}
\mathsf{k}^{-1}\mathsf{k} =\mathbf{1},\qquad\mathsf{kk}^{-1}=\mathbf{1},\notag\\
\mathsf{ke} =q^2\mathsf{ek},\label{ke}\\
\mathsf{kf} =q^{-2}\mathsf{fk},\label{kf}\\
\mathsf{ef}-\mathsf{fe} =\frac{\mathsf{k}-\mathsf{k}^{-1}}{q-q^{-1}}. \label{effe}
\end{gather}

The standard Hopf algebra structure on $U_q(\mathfrak{sl}_2)$ is determined by the comultiplication $\Delta$, the counit $\boldsymbol\varepsilon$, and the antipode $\mathsf{S}$ as follows
\begin{alignat}{4}
&\Delta(\mathsf{k})=\mathsf{k}\otimes\mathsf{k}, && && & \label{k0}\\
& \Delta(\mathsf{e}) =\mathbf{1}\otimes\mathsf{e}+ \mathsf{e}\otimes\mathsf{k}, && && & \label{def}\\
& \Delta(\mathsf{f}) =\mathsf{f}\otimes\mathbf{1}+\mathsf{k}^{-1}\otimes\mathsf{f},\qquad && && & \label{def1}\\
& \mathsf{S}(\mathsf{k}) =\mathsf{k}^{-1}, && \mathsf{S}(\mathsf{e}) =-\mathsf{ek}^{-1}, \qquad && \mathsf{S}(\mathsf{f}) =-\mathsf{kf}, &\notag\\
& \boldsymbol{\varepsilon}(\mathsf{k}) =\mathbf{1}, && \boldsymbol{\varepsilon}(\mathsf{e}) =\boldsymbol{\varepsilon}(\mathsf{f})=0.\qquad && &\notag
\end{alignat}

Given a $U_q(\mathfrak{sl}_2)$-symmetry on $\mathbb{C}_q\big[x^{\pm 1},y^{\pm 1}\big]$, the generator $\mathsf{k}$ acts via an automorphism of $\mathbb{C}_q\big[x^{\pm 1},y^{\pm 1}\big]$, as one can readily deduce from invertibility of $\mathsf{k}$, Definition \ref{symdef}(i) and \eqref{k0}. In particular, every symmetry determines uniquely a matrix $\sigma\in {\rm SL}(2,\mathbb{Z})$ as in \eqref{Auto}.

There exists a one-to-one correspondence between the $U_q(\mathfrak{sl}_2)$-symmetries of $\mathbb{C}_q\big[x^{\pm 1},y^{\pm 1}\big]$ that leave invariant the subalgebra $\mathbb{C}_q[x,y]$ and the $U_q(\mathfrak{sl}_2)$-symmetries on $\mathbb{C}_q[x,y]$. One can readily restrict such symmetry of $\mathbb{C}_q\big[x^{\pm 1},y^{\pm 1}\big]$ to $\mathbb{C}_q[x,y]$.

On the other hand, suppose we are given an arbitrary symmetry $\pi$ on
$\mathbb{C}_q\big[x^{\pm 1},y^{\pm 1}\big]$ (not necessarily leaving invariant
$\mathbb{C}_q[x,y]$). One has the following relations:
\begin{alignat}{3}
&\pi(\mathsf{k})\big(x^{-1}\big)=(\pi(\mathsf{k})x)^{-1}, & & \pi(\mathsf{k})\big(y^{-1}\big)=(\pi(\mathsf{k})y)^{-1}, & \label{ext_k}\\
& \pi(\mathsf{e})\big(x^{-1}\big)=-x^{-1}(\pi(\mathsf{e})x)(\pi(\mathsf{k})x)^{-1}, & & \pi(\mathsf{e})\big(y^{-1}\big) =-y^{-1}(\pi(\mathsf{e})y)(\pi(\mathsf{k})y)^{-1}, & \label{ext_e}\\
& \pi(\mathsf{f})\big(x^{-1}\big)=-\big(\pi\big(\mathsf{k}^{-1}\big)x\big)^{-1}(\pi(\mathsf{f})x)x^{-1}, \ \ \, & & \pi(\mathsf{f})\big(y^{-1}\big)=-\big(\pi\big(\mathsf{k}^{-1}\big)y\big)^{-1}(\pi(\mathsf{f})y)y^{-1}.\!\!\!\! &\label{ext_f}
\end{alignat}
Here \eqref{ext_k} is straightforward since $\pi(\mathsf{k})$ is an automorphism; \eqref{ext_e} and \eqref{ext_f} are derivable by `differentiating' (i.e., applying $\mathsf{e}$ and $\mathsf{f}$, respectively, to) \eqref{Lpr}. Certainly, these relations remain true when $x$ or $y$ is replaced by an arbitrary invertible element.

Thus, given a symmetry on $\mathbb{C}_q[x,y]$, the relations \eqref{ext_k}--\eqref{ext_f} determine a well-defined extension of it to the additional generators $x^{-1}$, $y^{-1}$, hence to $\mathbb{C}_q\big[x^{\pm 1},y^{\pm 1}\big]$.

\section[The symmetries with non-trivial $\sigma$ and the generic symmetries]{The symmetries with non-trivial $\boldsymbol{\sigma}$\\ and the generic symmetries}\label{nw}

We first reproduce the results of \cite{S} which present a partial list of $U_q(\mathfrak{sl}_2)$-symmetries on $\mathbb{C}_q\big[x^{\pm 1},y^{\pm 1}\big]$.

Here and in what follows we describe the (series of) $U_q(\mathfrak{sl}_2)$-symmetries on $\mathbb{C}_q\big[x^{\pm 1},y^{\pm 1}\big]$ via determining an action of the distinguished generators of $U_q(\mathfrak{sl}_2)$ on the generators of $\mathbb{C}_q\big[x^{\pm 1}{,}y^{\pm 1}\big]$. To derive the associated $U_q(\mathfrak{sl}_2)$-symmetry, we first extend the action to monomials (both in~$U_q(\mathfrak{sl}_2)$ and $\mathbb{C}_q\big[x^{\pm 1},y^{\pm 1}\big]$) using
\begin{gather*}
(\mathsf{ab})u \overset{\rm def}{=}\mathsf{a}(\mathsf{b}u), \qquad \mathsf{a},\mathsf{b} \in U_q(\mathfrak{sl}_2), \qquad u \in\mathbb{C}_q\big[x^{\pm 1},y^{\pm 1}\big],\\
 \mathsf{a}(uv) \overset{\rm def}{=} \sum_{(\mathsf{a})}(\mathsf{a}_{(1)}u)\cdot(\mathsf{a}_{(2)}v), \qquad \mathsf{a} \in U_q(\mathfrak{sl}_2), \qquad u,v \in\mathbb{C}_q\big[x^{\pm 1},y^{\pm 1}\big],
\end{gather*}
(the Sweedler sigma-notation is implicit here), and then extend by linearity to the entire algebras $U_q(\mathfrak{sl}_2)$ and $\mathbb{C}_q\big[x^{\pm 1},y^{\pm 1}\big]$, using
\begin{gather*}
\mathsf{a}(u+v)=\mathsf{a}u+\mathsf{a}v,\qquad (\mathsf{a}+\mathsf{b})u=\mathsf{a}u+\mathsf{b}u,\\
\mathbf{1}u=u,\qquad \mathsf{a}\mathbf{1}=\boldsymbol{\varepsilon}(\mathsf{a})\mathbf{1}, \qquad\mathsf{a},\mathsf{b}\in U_q(\mathfrak{sl}_2),\qquad u,v\in\mathbb{C}_q\big[x^{\pm 1},y^{\pm 1}\big].
\end{gather*}
Such extension determines a well-defined action of $U_q(\mathfrak{sl}_2)$ on $\mathbb{C}_q\big[x^{\pm 1},y^{\pm 1}\big]$ if and only if everything passes through the relations in $U_q(\mathfrak{sl}_2)$ and $\mathbb{C}_q\big[x^{\pm 1},y^{\pm 1}\big]$. To verify this, one has to apply every generator of $U_q(\mathfrak{sl}_2)$ to each relation in $\mathbb{C}_q\big[x^{\pm 1},y^{\pm 1}\big]$, and then every relation in $U_q(\mathfrak{sl}_2)$ to each generator of $\mathbb{C}_q\big[x^{\pm 1},y^{\pm 1}\big]$. This is to be done in each specific case, and normally such verification is left to the reader.

It turns out that there exists a single rather poor family (series) of $U_q(\mathfrak{sl}_2)$-symmetries on $\mathbb{C}_q\big[x^{\pm 1},y^{\pm 1}\big]$
that correspond to non-trivial matrices $\sigma\in {\rm SL}(2,\mathbb{Z})$ which determine the action of Cartan generator. More precisely, the set of matrices $\sigma\ne I$ which appear this way reduces to a single matrix $\sigma=-I$, and one has

\begin{Theorem}[{\cite[Theorem 3.5, type (I)]{S}}]\label{sigma=-I}
There exists a two-parameter $(\alpha,\beta\in\mathbb{C}{\setminus}\{0\})$ family of $U_q(\mathfrak{sl}_2)$-symmetries on $\mathbb{C}_q\big[x^{\pm 1},y^{\pm 1}\big]$ that correspond to $\sigma=-I$:
\begin{alignat*}{3}
&\pi(\mathsf{k})(x)=\alpha^{-1}x^{-1}, \qquad && \pi(\mathsf{k})(y) =\beta^{-1}y^{-1}, &\\
& \pi(\mathsf{e})(x) =0,\qquad & & \pi(\mathsf{e})(y) =0, &\\
& \pi(\mathsf{f})(x) =0,\qquad && \pi(\mathsf{f})(y) =0. &
\end{alignat*}
These are all the symmetries with $\sigma=-I$ $($and also with $\sigma\ne I)$. These symmetries are all isomorphic, in particular to that with $\alpha=\beta=1$.
\end{Theorem}

Now turn to the case $\sigma=I$ and let $\pi$ be a symmetry with this property. Then it follows from~\eqref{Auto} that the action of Cartan element~$\mathsf{k}$ is given by multiplication of the generators~$x$,~$y$ by non-zero {\it weight constants}. We denote these weight constants by $\alpha$ and $\beta$, respectively:
\begin{gather*}\pi(\mathsf{k})(x)=\alpha x,\qquad\pi(\mathsf{k})(y)=\beta y.\end{gather*}
Certainly, monomials form a basis of weight vectors (eigenvectors for $\pi(\mathsf{k})$), and the associated eigenvalues are called {\it weights}.

Let us start with the simplest case in which the operators $\pi(\mathsf{e})$ and $\pi(\mathsf{f})$ are identically zero. This case has been disregarded in \cite{S} due to its triviality, but should be considered now in order to get a complete list of symmetries.

\begin{Lemma}\label{iz}Let $\pi$ be a $U_q(\mathfrak{sl}_2)$-symmetry on $\mathbb{C}_q\big[x^{\pm 1},y^{\pm 1}\big]$. The following properties of $\pi$ are equivalent:
\begin{enumerate}\itemsep=0pt
\item[$(i)$] the weight constants $\alpha,\beta\in\{-1;1\}$;
\item[$(ii)$] $\pi(\mathsf{e})$ is the identically zero operator on $\mathbb{C}_q\big[x^{\pm 1},y^{\pm 1}\big]$;
\item[$(iii)$] $\pi(\mathsf{f})$ is the identically zero operator on $\mathbb{C}_q\big[x^{\pm 1},y^{\pm 1}\big]$;
\item[$(iv)$] both $\pi(\mathsf{e})$ and $\pi(\mathsf{f})$ are the identically zero operators on $\mathbb{C}_q\big[x^{\pm 1},y^{\pm 1}\big]$.
\end{enumerate}
\end{Lemma}

\begin{proof} Assume (i). Clearly the weight of any monomial is $\pm 1$. On the other hand, it follows from~\eqref{ke} that $\pi(\mathsf{e})$ takes $x$ to a weight vector whose weight is $\pm q^2\ne\pm 1$. Hence $\pi(\mathsf{e})(x)=0$, and one has also $\pi(\mathsf{e})\big(x^{-1}\big)=0$ in view of~\eqref{ext_e}. Similarly, $\pi(\mathsf{e})\big(y^{\pm 1}\big)=0$ in view of~\eqref{ke} and~\eqref{ext_e}. Thus we conclude that $\pi(\mathsf{e})\equiv 0$, which is just~(ii). The proof of (i) $\Rightarrow$ (iii) is similar.

Assume (ii). An application of~\eqref{effe} to $x$ and $y$ yields $\big(\pi(\mathsf{k})-\pi\big(\mathsf{k}^{-1}\big)\big)(x)=
\big(\pi(\mathsf{k})-\pi\big(\mathsf{k}^{-1}\big)\big)(y)=0$, hence $\alpha=\alpha^{-1}$, $\beta=\beta^{-1}$, which is equivalent to (i). The proof of (iii)~$\Rightarrow$~(i) is similar, and the rest of implications are clear. \end{proof}

\begin{Theorem}[type (II)]\label{T01}
Suppose that the weight constants $\alpha,\beta\in\{-1;1\}$. There exist $4$ $U_q(\mathfrak{sl}_2)$-symmetries on $\mathbb{C}_q\big[x^{\pm 1},y^{\pm 1}\big]$ given by
\begin{gather}
\pi(\mathsf{k})(x) =\pm x,\qquad\pi(\mathsf{k})(y)=\pm y, \label{t1}\\
\pi(\mathsf{e})(x) =\pi(\mathsf{e})(y)=\pi(\mathsf{f})(x)=\pi(\mathsf{f})(y)=0. \label{t2}
\end{gather}
These are all the symmetries with $\alpha,\beta\in\{-1;1\}$. They split in $2$ isomorphism classes: the first one is just the symmetry with $\alpha=\beta=1$ and the second one is the $($union of$)$ $3$ other symmetries.
\end{Theorem}

\noindent{\bf Proof} reduces to a (trivial) verification of the fact that~\eqref{t1} and~\eqref{t2} determine well-defined $U_q(\mathfrak{sl}_2)$-symmetries on $\mathbb{C}_q\big[x^{\pm 1},y^{\pm 1}\big]$, with a subsequent application of Lemma~\ref{iz}. The isomorphisms between the symmetries with $\binom{\alpha}{\beta}\ne\binom{1}{1}$ are easily available.
\hfill \qed

\medskip

In the rest of our considerations we assume that either $\pi(\mathsf{e})$ or $\pi(\mathsf{f})$ is not identically zero.

A pair of non-zero complex constants $\alpha$ and $\beta$ which could appear as weight constants for some $U_q(\mathfrak{sl}_2)$-symmetry of $\mathbb{C}_q\big[x^{\pm 1},y^{\pm 1}\big]$, can not be arbitrary. In fact, an obvious consequence of~\eqref{ke} claims that $\pi(\mathsf{e})$ sends a vector whose weight is $\gamma$ to a vector whose weight is $q^2\gamma$. In particular, $\pi(\mathsf{e})(x)$, if non-zero, is a sum of monomials with (the same) weight $q^2\alpha$. Since the weight of the monomial $ax^iy^j$ (with $a\ne 0$) is $\alpha^i\beta^j$, one has that $\alpha^u\beta^v=q^2$ for some integers~$u$,~$v$. Of course similar conclusions can be also derived by applying \eqref{ke}, \eqref{kf} to~$x$ and~$y$. Under our present assumptions (contrary to those of Lemma~\ref{iz}), this argument demonstrates that the pair~$\alpha$ and $\beta$ of weight constants for a $U_q(\mathfrak{sl}_2)$-symmetry of $\mathbb{C}_q\big[x^{\pm 1},y^{\pm 1}\big]$ is subject to $\alpha^u\beta^v=q^2$ for some (in general, non-unique) integers~$u$,~$v$.

The following theorem covers all but a countable family of admissible pairs of weight constants. The series of symmetries involved here have been called {\it generic} in~\cite{S}.

\begin{Theorem}[{\cite[Theorem 4.1, type (III)]{S}}]\label{gener}
Let a pair of constants $\alpha,\beta\in\mathbb{C}{\setminus}\{0\}$ be such that $\alpha^u\beta^v=q^2$ for some $u,v\in\mathbb{Z}$ and $\alpha^m\beta^n\ne 1$ for all integers $m$, $n$ with $(m,n)\ne(0,0)$. Then there exists a one-parameter $(a\in\mathbb{C}{\setminus}\{0\})$ family of $U_q(\mathfrak{sl}_2)$-symmetries of $\mathbb{C}_q\big[x^{\pm 1},y^{\pm 1}\big]$:
\begin{alignat}{3}
&\pi(\mathsf{k})(x)=\alpha x, \qquad & &\pi(\mathsf{k})(y) =\beta y, & \label{gener_k}\\
& \pi(\mathsf{e})(x) =aq^{uv+3}\frac{1-\alpha q^v}{\big(1-q^2\big)^2}x^{u+1}y^v, \qquad & & \pi(\mathsf{e})(y) =aq^{uv+3}\frac{q^u-\beta}{\big(1-q^2\big)^2}x^uy^{v+1}, & \label{gener_e}\\
& \pi(\mathsf{f})(x) =-\frac{\big(\alpha^{-1}-q^{-v}\big)}{a}x^{-u+1}y^{-v},\qquad && \pi(\mathsf{f})(y) =-\frac{\big(\beta^{-1}q^{-u}-1\big)}{a}x^{-u}y^{-v+1}. &\label{gener_f}
\end{alignat}
There exist no other symmetries with the weight constants $\alpha$ and $\beta$. Two symmetries of this form are isomorphic if and only if the associated pairs of weight constants are on the same orbit of the ${\rm SL}(2,\mathbb{Z})$-action \eqref{sl_ac} on $(\mathbb{C}{\setminus}\{0\})^2$.
\end{Theorem}

\begin{proof} We restrict ourselves to proving the last claim (the isomorphism criterion) which was absent in \cite[Theorem 4.1]{S}.

The `only if' part is obvious.

Assume that two symmetries $\pi$ and $\pi'$ as in \eqref{gener_k}--\eqref{gener_f} have pairs of weight constants which are on the same orbit of the ${\rm SL}(2,\mathbb{Z})$-action. Let us replace one of these symmetries (let it be~$\pi$) with an isomorphic symmetry $B\pi B^{-1}$, where $B=\varphi_{\sigma,1,1}\in\operatorname{Aut}\big(\mathbb{C}_q\big[x^{\pm 1},y^{\pm 1}\big]\big)$ is as in~\eqref{Auto} for a suitable matrix $\sigma\in {\rm SL}(2,\mathbb{Z})$. This replacement allows us to assume that $\pi$ and $\pi'$ have the same pair of weight constants and differ only in the complex constants $a$ and $a'$, respectively, in~\eqref{gener_e},~\eqref{gener_f}. Now with $A=\varphi_{I,\mu,\nu}\in\operatorname{Aut}\big(\mathbb{C}_q\big[x^{\pm 1},y^{\pm 1}\big]\big)$ one can arrange a routine computation of $A\pi(\mathsf{e})A^{-1}(x)$, $A\pi(\mathsf{e})A^{-1}(y)$, $A\pi(\mathsf{f})A^{-1}(x)$, $A\pi(\mathsf{f})A^{-1}(y)$ based on~\eqref{gener_e},~\eqref{gener_f}, and~\eqref{Auto} (left to the reader) and then equate these $4$ elements of $\mathbb{C}_q\big[x^{\pm 1},y^{\pm 1}\big]$ to $\pi'(\mathsf{e})(x)$, $\pi'(\mathsf{e})(y)$, $\pi'(\mathsf{f})(x)$, $\pi'(\mathsf{f})(y)$, respectively. Using the fact that, under the assumptions of the Theorem, none of the constant multipliers in~\eqref{gener_e}, \eqref{gener_f} can be zero, one concludes that each of the $4$ relations is equivalent to the same relation $\mu^u\nu^v=\frac{a'}a$. As one can readily deduce from the assumptions on weight constants that $(u,v)\ne(0,0)$, a pair~$(\mu,\nu)$ satisfying the latter relation exists. This proves our claim.
\end{proof}

\begin{Remark}The generic symmetries described by Theorem~\ref{gener} constitute the most massive part of the entire collection of $U_q(\mathfrak{sl}_2)$-symmetries of $\mathbb{C}_q\big[x^{\pm 1},y^{\pm 1}\big]$, as compared with those given by Theorems~\ref{sigma=-I} and~\ref{T01} where $\pi(\mathsf{e})$ and $\pi(\mathsf{f})$ are the identically zero operators on $\mathbb{C}_q\big[x^{\pm 1},y^{\pm 1}\big]$. Even more, among the symmetries with $\sigma=I$ where the pair of weight constants $(\alpha,\beta)$ is the most essential feature of a specific symmetry, the collection of generic symmetries covers all but a countable family of admissible pairs of weight constants. This is contrary to non-generic symmetries described below in this work, where the collection of pairs of weight constants involved is only countable. This observation is completed by
\end{Remark}

\begin{Proposition}\label{ufnis} The collection of generic symmetries presented by Theorem~{\rm \ref{gener}} contains an uncountable family of isomorphism classes of $U_q(\mathfrak{sl}_2)$-symmetries on $\mathbb{C}_q\big[x^{\pm 1},y^{\pm 1}\big]$.
\end{Proposition}

\begin{proof} Let a generic symmetry $\pi$ with weight constants $(\alpha,\beta)$ be given. Consider an isomorphic symmetry $\pi'=A^{-1}\pi A$, where $A=\varphi_{\sigma,\mu,\nu}\in\operatorname{Aut}\big(\mathbb{C}_q\big[x^{\pm 1},y^{\pm 1}\big]\big)$ is as in~\eqref{Auto}. A routine verification demonstrates that the weight constants $\alpha'$, $\beta'$ of $\pi'$ are given by $\binom{\alpha'}{\beta'}=\sigma\binom{\alpha}{\beta}$ with the action of $\sigma$ being as in~\eqref{sl_ac}, while the parameters $\mu$, $\nu$ do not affect the weight constants of the isomorphic symmetry $\pi'$.

The first consequence of the above argument is that the weight constants $\alpha'$, $\beta'$ of $\pi'$ are again subject to the assumptions of Theorem \ref{gener}, hence $\pi'$ is also generic. That is, the family of generic symmetries is closed under passage to an isomorphic symmetry.

Another thing to be deduced is that every isomorphism class of generic symmetries is contained in a certain class of symmetries whose pairs of weight constants are on the same orbit of the action of (the countable group) ${\rm SL}(2,\mathbb{Z})$ on an uncountable invariant subset of $(\mathbb{C}{\setminus}\{0\})^2$. The latter subset is just the collection of $(\alpha,\beta)$ subject to the assumptions of Theorem \ref{gener}. This implies our claim.
\end{proof}

\begin{Remark}It has been mentioned in Section \ref{prel} that there exists a one-to-one correspondence between the $U_q(\mathfrak{sl}_2)$-symmetries on $\mathbb{C}_q[x,y]$ appearing in \cite[Theorems 4.2--4.7]{DS} and those of $\mathbb{C}_q\big[x^{\pm 1},y^{\pm 1}\big]$ that leave invariant the subalgebra $\mathbb{C}_q[x,y]$, which is nothing more than the restriction-extension procedures. In particular, every $U_q(\mathfrak{sl}_2)$-symmetry on $\mathbb{C}_q[x,y]$ described in~\cite{DS} should have its counterpart in the list of symmetries we produce here. It is easily visible that, in view of the above correspondence, the symmetries of \cite[Theorem 4.2]{DS} are the same as the symmetries in Theorem \ref{T01}.

As for all other symmetries of \cite{DS}, they do not extend to symmetries as in Theorems~\ref{sigma=-I} and~\ref{gener}. In the first case, the symmetries with $\sigma=-I$ do not leave $\mathbb{C}_q[x,y]$ invariant. As for the generic symmetries, it should be noted that, under the assumptions of Theorem~\ref{gener}, the constant multipliers in~\eqref{gener_e},~\eqref{gener_f} are all (including fractions) non-zero. On the other hand, one can find in every non-trivial symmetry in the final table of \cite{DS}, at least one identically zero expression among the formulas for $\pi(\mathsf{e})(x)$, $\pi(\mathsf{e})(y)$, $\pi(\mathsf{f})(x)$, $\pi(\mathsf{f})(y)$.

Hence, the generic symmetries do not leave invariant the subalgebra $\mathbb{C}_q[x,y]$. Another thing to be deduced here is that the counterparts of non-trivial symmetries~\cite{DS} are to be found among the non-generic symmetries of $\mathbb{C}_q\big[x^{\pm 1},y^{\pm 1}\big]$ described below.
\end{Remark}

\section{Weight polynomials and extreme monomials}\label{wpem}

We start with defining the integral parameters of type (IV) (non-generic) $U_q(\mathfrak{sl}_2)$-symmetries on $\mathbb{C}_q\big[x^{\pm 1},y^{\pm 1}\big]$ based on considering the properties of the associated pair of weight constants. Although these parameters are not determined uniquely by a symmetry, they allow one to derive two more parameters $D$ and $G$. The subsequent exposition demonstrates that the latter para\-me\-ters, while still admit certain ambiguities in their definition, work as isomorphism invariants for type (IV) symmetries and allow the classification suggested in this work.

In what follows we make implicit the natural action of (the semigroup of) integral matrices by endomorphisms of the multiplicative group $(\mathbb{C}{\setminus}\{0\})^2$:
\begin{gather}
\begin{pmatrix}a & b\\ c & d\end{pmatrix}\binom\mu\nu=\binom{\mu^a\nu^b}{\mu^c\nu^d}. \label{*}
\end{gather}

{\bf Integral parameters of a non-generic symmetry.} Let $\pi$ be a $U_q(\mathfrak{sl}_2)$-symmetry on $\mathbb{C}_q\big[x^{\pm 1},y^{\pm 1}\big]$ with the associated matrix $\sigma=I$. Assume that $\pi$ is type (IV), that is, the associated pair of weight constants $\alpha$, $\beta$ satisfy $\alpha^u\beta^v=q^2$ for some $u,v\in\mathbb{Z}$ (just the assumption on~$\pi$ being not type~(II)), but, on the contrary to assumptions of Theorem \ref{gener}, $\alpha^r\beta^s=1$ for some integers $r,s$ other than $r=s=0$. Let us choose a pair $(r,s)$ to be {\it minimal} in the following sense.

\begin{Definition}Given a non-generic $U_q(\mathfrak{sl}_2)$-symmetry $\pi$ on $\mathbb{C}_q\big[x^{\pm 1},y^{\pm 1}\big]$ with the associated matrix $\sigma=I$ and the weight constants $\alpha$, $\beta$, the pair of integers $(r,s)$ with $\alpha^r\beta^s=1$ is said to be {\it minimal} if, under the assumption $r=wr'$, $s=ws'$ for some integers $w$, $r'$, $s'$, $w\ne\pm 1$, one has $\alpha^{r'}\beta^{s'}\ne 1$.
\end{Definition}

Clearly, under these settings the pair $(r,s)$ is unique up to multiplying both $r$ and $s$ by $-1$, and we assume this pair to be fixed while considering a specific type~(IV) $U_q(\mathfrak{sl}_2)$-symmetry $\pi$.

Also, the pair $(u,v)$ is not unique for the given $\pi$, as with $u'=u+wr$, $v'=v+ws$, $w\in\mathbb{Z}$, one has $\alpha^{u'}\beta^{v'}=q^2$.

Let us associate to a type (IV) symmetry $\pi$ the matrix $\Phi=\left(\begin{smallmatrix}r & s\\ u & v\end{smallmatrix}\right)$, then
the weight constants $\alpha$, $\beta$ satisfy
\begin{gather}\label{wce}
\Phi\binom{\alpha}{\beta}=\binom{1}{q^2}.
\end{gather}
The entries $u$, $v$, $r$, $s$ of $\Phi$ will be referred to as {\it integral parameters} of a type (IV) symmetry $\pi$, with the minimality for the pair $(r,s)$ being implicit. Now~\eqref{wce} may be treated as a definition (although ambiguous in view of the above discussion) of integral parameters of $\pi$.

One may ask if an arbitrary integral matrix corresponds in a manner described above to a non-generic $U_q(\mathfrak{sl}_2)$-symmetry $\pi$ on $\mathbb{C}_q\big[x^{\pm 1},y^{\pm 1}\big]$. Our subsequent observations demonstrate that this conjecture fails.

One more trivial consequence of our choices is in observing that for any integers $m$, $n$, the subspace of weight polynomials with the same weight as $x^my^n$ is just the linear span of the monomials $x^{m+wr}y^{n+ws}$, $w\in\mathbb{Z}$.

Let us introduce the {\it discriminant} $D=rv-su$ of a non-generic symmetry $\pi$. It follows from our assumptions that $D\ne 0$. The above ambiguities in the choice of $u$, $v$, $r$, $s$, given $\pi$, could at most affect $D$ in multiplication by $-1$. This corresponds to the context of

\begin{Proposition}\label{D_inv}$|D|$ is an invariant of isomorphism class of non-generic $U_q(\mathfrak{sl}_2)$-symmetries on $\mathbb{C}_q\big[x^{\pm 1},y^{\pm 1}\big]$.
\end{Proposition}

\begin{proof} Let a non-generic symmetry $\pi$ be given. Let us consider an isomorphic symmetry $\pi'=\Psi^{-1}\pi\Psi$, with an automorphism $\Psi\in\operatorname{Aut}\big(\mathbb{C}_q\big[x^{\pm 1},y^{\pm 1}\big]\big)$, $\Psi=\varphi_{\sigma,\mu,\nu}$ as in \eqref{Auto}, whose associated matrix is $\sigma=\left(\begin{smallmatrix}k & m\\ l & n\end{smallmatrix}\right)\in {\rm SL}(2,\mathbb{Z})$. Clearly, the weight constants of $\pi'$ are $\alpha'=\alpha^k\beta^m$, $\beta'=\alpha^l\beta^n$, that is $\sigma\binom\alpha\beta=\binom{\alpha'}{\beta'}$.

With the matrix $\Phi=\left(\begin{smallmatrix}r & s\\ u & v\end{smallmatrix}\right)$ of integral parameters of $\pi$ one has
\begin{gather*}
\Phi\sigma^{-1}\binom{\alpha'}{\beta'}=\Phi\binom\alpha\beta=\binom{1}{q^2}.
\end{gather*}
This implies that the integral matrix $\Phi'\stackrel{\mathrm{def}}{=}\Phi\sigma^{-1}=\left(\begin{smallmatrix}r' & s'\\ u' & v'\end{smallmatrix}\right)$ is formed by the integral parameters of the symmetry $\pi'$ subject to all the necessary relations. Among those, the only item to be verified is the minimality condition for $r'$, $s'$.

Assume the contrary, that is $r'=wr''$, $s'=ws''$ for some non-zero integer $w\ne\pm 1$, but $\alpha^{r''}\beta^{s''}=1$. Consider the matrix $\Phi''\stackrel{\mathrm{def}}{=}\left(\begin{smallmatrix}r'' & s''\\ u' & v'\end{smallmatrix}\right)$, then one has
\begin{gather*}
\Phi''\binom{\alpha'}{\beta'}=\binom{1}{q^2}\qquad\text{and}\qquad
\Phi'=\Phi\sigma^{-1}=\begin{pmatrix}w & 0\\ 0 & 1\end{pmatrix}\Phi''.
\end{gather*}
It follows that $\left(\begin{smallmatrix}w^{-1} & 0\\ 0 & 1\end{smallmatrix}\right)\Phi=\Phi''\sigma$ is an integral matrix, and the latter relation being applied to $\binom\alpha\beta$ yields
\begin{gather*}
\begin{pmatrix}w^{-1} & 0\\ 0 & 1\end{pmatrix}\Phi\binom\alpha\beta=\Phi''\sigma\binom\alpha\beta=
\Phi''\binom{\alpha'}{\beta'}=\binom{1}{q^2}.
\end{gather*}
This contradicts to the minimality assumption on $r$, $s$, as the latter are just the entries of the first line of $\Phi$. The contradiction we get this way proves the minimality assumption on~$r'$,~$s'$.

Clearly, $\det\Phi'=\det\Phi$, which implies the claim of our proposition. \end{proof}

\begin{Remark}The proof of Proposition \ref{D_inv} could persuade the reader that the discriminant $D=\det\Phi=rv-su$ is a more reasonable and subtle invariant than $|D|$, which might possibly separate more isomorphism classes of non-generic symmetries. However, it will become clear later that (simultaneous) multiplying $r$ and $s$ (hence also $\det\Phi$) by~$-1$ produces a sort of reparametrization of the same symmetry, thus remaining intact its isomorphism class.
\end{Remark}

There exists one more invariant, to be used later in classifying the symmetries. Let $\operatorname{gcd}(r,s)$ stand for the greatest common divisor for integers $r$, $s$.

\begin{Proposition}$G=\operatorname{gcd}(r,s)$ is an invariant of isomorphism class of non-generic $U_q(\mathfrak{sl}_2)$-symmetries on $\mathbb{C}_q\big[x^{\pm 1},y^{\pm 1}\big]$, with $r$, $s$ being the matrix elements of the first line of $\Phi$, the matrix of integral parameters of a symmetry.
\end{Proposition}

\begin{proof} We need to reproduce the beginning of proof of Proposition \ref{D_inv}. Let us consider again two isomorphic symmetries $\pi$ and $\pi'=\Psi^{-1}\pi\Psi$, with $\Psi\in\operatorname{Aut}\big(\mathbb{C}_q\big[x^{\pm 1},y^{\pm 1}\big]\big)$, $\Psi=\varphi_{\sigma,\mu,\nu}$, with the associated matrix $\sigma=\left(\begin{smallmatrix}k & m\\ l & n\end{smallmatrix}\right)\in {\rm SL}(2,\mathbb{Z})$, so that $\sigma\binom\alpha\beta=\binom{\alpha'}{\beta'}$, with $\binom\alpha\beta$, (respectively, $\binom{\alpha'}{\beta'}$) being the weight constants of~$\pi$ (respectively,~$\pi'$).

Let $\Phi=\left(\begin{smallmatrix}r & s\\ u & v\end{smallmatrix}\right)$ be the matrix of integral parameters for $\pi$. One has
\begin{gather*}
\Phi\sigma^{-1}\binom{\alpha'}{\beta'}=\Phi\binom\alpha\beta=\binom{1}{q^2}.
\end{gather*}
That is, the integral matrix $\Phi'\stackrel{\mathrm{def}}{=}\Phi\sigma^{-1}=\left(\begin{smallmatrix}r' & s'\\ u' & v'\end{smallmatrix}\right)$ is formed by the integral parameters of the symmetry $\pi'$ subject to all the necessary relations. One can now verify the minimality condition for $r'$, $s'$ exactly as it was done in the proof of Proposition \ref{D_inv}.

Let $w=\operatorname{gcd}(r,s)$, then $r=wr_1$, $s=ws_1$, with $r_1$, $s_1$ being coprime. One has $\Phi=\left(\begin{smallmatrix}w & 0\\ 0 & 1\end{smallmatrix}\right)\Phi_1'$, where $\Phi_1'=\left(\begin{smallmatrix}r_1 & s_1\\ u & v\end{smallmatrix}\right)$.

It follows that
\begin{gather*}
\Phi'=\Phi\sigma^{-1}=\begin{pmatrix}w & 0\\ 0 & 1\end{pmatrix}\Phi_1'\sigma^{-1},
\end{gather*}
whence $\operatorname{gcd}(r',s')\ge w=\operatorname{gcd}(r,s)$.

One can readily get the opposite inequality via applying this argument to the reverse passage from $\pi'$ to $\pi$. Thus, we finally conclude that $\operatorname{gcd}(r',s')=\operatorname{gcd}(r,s)$.
\end{proof}

\begin{Proposition}A $($non-degenerate$)$ integral matrix $\Phi=\left(\begin{smallmatrix}r & s\\ u & v\end{smallmatrix}\right)$ of integral parameters for a~type~$(IV)$ symmetry~$\pi$ determines $($via \eqref{*}$)$ an onto endomorphism of the multiplicative group $(\mathbb{C}{\setminus}\{0\})^2$. The equation~\eqref{wce} has finitely many solutions. Every solution $\binom{\alpha}{\beta}$, in terms of the matrix elements $u$, $v$, $r$, $s$ and the discriminant $D=\det\Phi$, is subject to
\begin{gather}\label{abgf}
\alpha^D=q^{-2s},\qquad\beta^D=q^{2r}.
\end{gather}
\end{Proposition}

\begin{proof} The Cramer theorem allows one to form an integral matrix $\Phi'$ such that $\Phi\Phi'=\Phi'\Phi=\left(\begin{smallmatrix}D & 0\\ 0 & D\end{smallmatrix}\right)$, where $D=\det\Phi=rv-su\ne 0$. Explicitly $\Phi'=\left(\begin{smallmatrix}v & -s\\ -u & r\end{smallmatrix}\right)$. Given arbitrary $\binom{\mu}{\nu}\in(\mathbb{C}{\setminus}\{0\})^2$, one chooses $\mu'$, $\nu'$ with $\mu'^D=\mu$, $\nu'^D=\nu$. Then $\Phi'\binom{\mu'}{\nu'}$ is a preimage for $\binom{\mu}{\nu}$ with respect to the endomorphism determined by $\Phi$, because
\begin{gather*}
\Phi\Phi'\binom{\mu'}{\nu'}=\begin{pmatrix}D & 0\\ 0 & D\end{pmatrix}\binom{\mu'}{\nu'}=\binom{\mu}{\nu}.
\end{gather*}
Hence the endomorphism in question is onto. Also, a routine argument establishes that the kernel of this map is contained in $\Gamma^2\subset(\mathbb{C}{\setminus}\{0\})^2$, the Cartesian square of the group $\Gamma$ of $D$-th roots of $1$, hence finite.

An application of $\Phi'$ to \eqref{wce} produces \eqref{abgf}.
\end{proof}

Let us consider a non-zero weight polynomial $p\in\mathbb{C}_q\big[x^{\pm 1},y^{\pm 1}\big]$ with respect to a symmet\-ry~$\pi$, whose integral parameters are $r$, $s$, $u$, $v$. Due to the above observations, $p$ is a finite sum of the form $p=\sum\limits_{w\in\mathbb{Z}}d_wx^{m+wr}y^{n+ws}$, so that the weight of $p$ is $\alpha^m\beta^n$. We need to consider its {\it extreme monomials}, namely the maximum and minimum ones, which correspond, respectively, to $\max\{w\colon d_w\ne 0\}$ and $\min\{w\colon d_w\ne 0\}$. The associated values of $w$ will be denoted by $\operatorname{maxind}_{m,n}(p)$ and $\operatorname{minind}_{m,n}(p)$. This certainly assumes that the pair $(m,n)$ is fixed.

Consider the polynomials $\pi(\mathsf{e})(x)$, $\pi(\mathsf{e})(y)$, $\pi(\mathsf{f})(x)$, $\pi(\mathsf{f})(y)$, related to the symmetry $\pi$. It follows from the commutation relations \eqref{ke}, \eqref{kf} that these are weight polynomials with weights $q^2\alpha$, $q^2\beta$, $q^{-2}\alpha$, $q^{-2}\beta$, respectively. Furthermore, the discussion at the beginning of this Section allows one to conclude that these should be finite sums of the form
\begin{alignat}{3}
&\pi(\mathsf{e})(x)=\sum_{w\in\mathbb{Z}}a_w'x^{u+1+wr}y^{v+ws}, \qquad && \pi(\mathsf{e})(y) =\sum_{w\in\mathbb{Z}}a_w''x^{u+wr}y^{v+1+ws}, & \label{exy1}\\
& \pi(\mathsf{f})(x) =\sum_{t\in\mathbb{Z}}c_t'x^{-u+1+tr}y^{-v+ts},\qquad && \pi(\mathsf{f})(y) =\sum_{t\in\mathbb{Z}}c_t''x^{-u+tr}y^{-v+1+ts}. & \label{fxy1}
\end{alignat}

It should be emphasized that the above notions of extreme monomials and the values \linebreak $\operatorname{maxind}_{m,n}(p)$ and $\operatorname{minind}_{m,n}(p)$ are well defined only for {\it non-zero} polynomials $p$. We need to consider similar notions in a way more closely related to a (non-generic) $U_q(\mathfrak{sl}_2)$-symmet\-ry~$\pi$ on $\mathbb{C}_q\big[x^{\pm 1},y^{\pm 1}\big]$. Namely, given such symmetry $\pi$, we consider the double (pair of) polynomial(s) $(\pi(\mathsf{e})(x),\pi(\mathsf{e})(y))\in\mathbb{C}_q\big[x^{\pm 1},y^{\pm 1}\big]\oplus\mathbb{C}_q\big[x^{\pm 1},y^{\pm 1}\big]$. This may be treated as a (finite) sum
\begin{gather*}
(\pi(\mathsf{e})(x),\pi(\mathsf{e})(y))=
\sum_{w\in\mathbb{Z}}\big(a_w'x^{u+1+wr}y^{v+ws},a_w''x^{u+wr}y^{v+1+ws}\big)
\end{gather*}
of double monomials with the coefficients $a_w'$, $a_w''$ as in \eqref{exy1}, \eqref{fxy1}.

Let us introduce the notation
\begin{gather*}
\operatorname{minind}_\mathsf{e}(\pi) =\min\{w\in\mathbb{Z}\,|\,(a_w',a_w'')\ne(0,0)\}, \\
 \operatorname{maxind}_\mathsf{e}(\pi) =\max\{w\in\mathbb{Z}\,|\,(a_w',a_w'')\ne(0,0)\}.
\end{gather*}

In a similar way we consider the double polynomial $(\pi(\mathsf{f})(x),\pi(\mathsf{f})(y))\in\mathbb{C}_q\big[x^{\pm 1},y^{\pm 1}\big]\oplus\mathbb{C}_q\big[x^{\pm 1},y^{\pm 1}\big]$ and introduce
\begin{gather*}
\operatorname{minind}_\mathsf{f}(\pi) =\min\{t\in\mathbb{Z}\,|\,(c_t',c_t'')\ne(0,0)\},\\
\operatorname{maxind}_\mathsf{f}(\pi) =\max\{t\in\mathbb{Z}\,|\,(c_t',c_t'')\ne(0,0)\}.
\end{gather*}

The four values $\operatorname{minind}_\mathsf{e}(\pi)$, $\operatorname{maxind}_\mathsf{e}(\pi)$, $\operatorname{minind}_\mathsf{f}(\pi)$, $\operatorname{maxind}_\mathsf{f}(\pi)$ are well defined when we are in the context of non-generic symmetries.

There exist certain dependencies between the constant coefficients in \eqref{exy1}, \eqref{fxy1}. To clarify them, we need the following lemma based on~\eqref{qpr}.

\begin{Lemma}[{\cite[Lemma 4.2]{S}}]\label{ratios} Let $\pi$ be a $U_q(\mathfrak{sl}_2)$-symmetry on $\mathbb{C}_q\big[x^{\pm 1},y^{\pm 1}\big]$ such that
\begin{alignat}{3}
&\pi(\mathsf{k})(x)=\alpha x,\qquad && \pi(\mathsf{k})(y) =\beta y, &\notag\\
& \pi(\mathsf{e})(x) =\sum_{i,j}a_{i,j}x^iy^j, \qquad && \pi(\mathsf{e})(y) =\sum_{i,j}b_{i,j}x^iy^j, &\label{ex_y}\\
& \pi(\mathsf{f})(x) =\sum_{i,j}c_{i,j}x^iy^j,\qquad && \pi(\mathsf{f})(y) =\sum_{i,j}d_{i,j}x^iy^j, &\label{fx_y}
\end{alignat}
with $\alpha,\beta\in\mathbb{C}{\setminus}\{0\}$, $a_{i,j},b_{i,j},c_{i,j},d_{i,j}\in\mathbb{C}$, and the above sums being finite. Then
\begin{gather*}
a_{i+1,j}\big(q^i-\beta\big) =b_{i,j+1}\big(1-\alpha q^j\big),\\
 c_{i+1,j}\big(1-\beta^{-1}q^i\big) =d_{i,j+1}\big(q^j-\alpha^{-1}\big).
\end{gather*}
\end{Lemma}

An application of Lemma \ref{ratios} allows one to establish the relations between $a_w'$, $a_w''$ and between~$c_w'$,~$c_w''$. Namely, we substitute $a_w'=a_w(1-\alpha q^{v+ws})$, $a_w''=a_w(q^{u+wr}-\beta)$ and rewrite~\eqref{exy1},~\eqref{fxy1} in the form
\begin{gather}
\pi(\mathsf{e})(x) =\sum_{w\in\mathbb{Z}}a_w\big(1-\alpha q^{v+ws}\big)x^{u+1+wr}y^{v+ws},\label{ex2}\\
\pi(\mathsf{e})(y) =\sum_{w\in\mathbb{Z}}a_w\big(q^{u+wr}-\beta\big)x^{u+wr}y^{v+1+ws}, \label{ey2}\\
\pi(\mathsf{f})(x) =\sum_{t\in\mathbb{Z}}c_t\big(\alpha^{-1}-q^{-v+ts}\big)x^{-u+1+tr}y^{-v+ts}, \label{fx2}\\
\pi(\mathsf{f})(y) =\sum_{t\in\mathbb{Z}}c_t\big(\beta^{-1}q^{-u+tr}-1\big)x^{-u+tr}y^{-v+1+ts}. \label{fy2}
\end{gather}

To clarify the relations between $a_w$ and $c_t$ required to write down the existing series of non-generic symmetries, we need to collect several properties of extreme monomials of the weight polynomials \eqref{ex2}--\eqref{fy2}.

\begin{Lemma}\label{ac_0}
\begin{alignat}{3}
&a_w=0 \qquad && \text{for} \quad w >\operatorname{maxind}_\mathsf{e}(\pi),&\label{awmax}\\
& a_w =0 \qquad &&\text{for} \quad w <\operatorname{minind}_\mathsf{e}(\pi),&\label{awmin}\\
& c_t =0 \qquad &&\text{for} \quad t >\operatorname{maxind}_\mathsf{f}(\pi),&\label{ctmax}\\
& c_t =0 \qquad && \text{for} \quad t <\operatorname{minind}_\mathsf{f}(\pi). &\label{ctmin}
\end{alignat}
\end{Lemma}

\begin{proof}To prove \eqref{awmax} and \eqref{awmin}, it suffices, in view of our definitions, to verify that the multipliers $1-\alpha q^{v+ws}$ in \eqref{ex2} and $q^{u+wr}-\beta$ in~\eqref{ey2} can not be zero simultaneously (with the same $w$).

Assuming the contrary, we get $\alpha=q^{-v-ws}$, $\beta=q^{u+wr}$ for some~$w$. Under our assumption on~$q$, this, being substituted into~\eqref{abgf}, implies
\begin{gather*}
D(u+wr) =2r, \qquad D(v+ws) =2s.
\end{gather*}
This can be readily rewritten in the form
\begin{gather*}
Du =(2-wD)r,\qquad Dv=(2-wD)s,
\end{gather*}
and since the discriminant $D\ne 0$, this implies $D=\det\Phi=0$. The contradiction we get this way proves our claim.

\eqref{ctmax}, \eqref{ctmin} can be proved in a similar way.
\end{proof}

One more application of the argument, used in the proof of Lemma \ref{ac_0}, together with our definitions, yields

\begin{Lemma}\label{ecnz}
In \eqref{ex2}, \eqref{ey2} one has
\begin{gather*}
a_{\operatorname{maxind}_\mathsf{e}(\pi)}\ne 0,\qquad a_{\operatorname{minind}_\mathsf{e}(\pi)}\ne 0.
\end{gather*}
In \eqref{fx2}, \eqref{fy2} one has
\begin{gather*}
c_{\operatorname{maxind}_\mathsf{f}(\pi)}\ne 0,\qquad c_{\operatorname{minind}_\mathsf{f}(\pi)}\ne 0.
\end{gather*}
\end{Lemma}

We also need the following relations valid in our present context ($\sigma=I$):
\begin{gather}
\pi(\mathsf{e})(x^p)=\sum_{i,j}a_{i,j}\frac{\alpha^pq^{jp}-1}{\alpha q^j-1}x^{p-1+i}y^j, \label{exp}\\
\pi(\mathsf{e})(y^p)=\sum_{i,j}b_{i,j}\frac{\beta^p-q^{ip}}{\beta-q^i}x^iy^{p-1+j}, \label{eyp}\\
\pi(\mathsf{f})(x^p)=\sum_{i,j}c_{i,j}\frac{\alpha^{-p}-q^{jp}}{\alpha^{-1}-q^j}x^{p-1+i}y^j, \label{fxp}\\
\pi(\mathsf{f})(y^p)=\sum_{i,j}d_{i,j}\frac{\beta^{-p}q^{ip}-1}{\beta^{-1}q^i-1}x^iy^{p-1+j}, \label{fyp}
\end{gather}
where $p\in\mathbb{Z}$, $a_{i,j},b_{i,j},c_{i,j},d_{i,j}\in\mathbb{C}$ are as in~\eqref{ex_y},~\eqref{fx_y}. These relations are due to a straightforward induction argument, together with~\eqref{ext_e},~\eqref{ext_f}. Of course, these are valid unless some special values for the weight constants $\alpha$ and/or $\beta$ make the denominators of fractions involved to be zero. In fact we will not encounter such special cases in what follows, so we need not take care about replacing these fractions (which are just sums of certain progressions) to attain more generality in \eqref{exp}--\eqref{fyp}.

We demonstrate here \eqref{exp} for the reader's convenience, assuming $\alpha q^j-1\ne 0$. Clearly with $p=0$ we get $\pi(\mathsf{e})(1)=0$ as one should expect. Substituting $p=1$, we get just~\eqref{ex_y}. With $p>0$, let us perform the induction step.
\begin{gather*}
\pi(\mathsf{e})\big(x^{p+1}\big)= x^p\pi(\mathsf{e})(x)+\pi(\mathsf{e})\big(x^p\big)\pi(\mathsf{k})(x)\\
\hphantom{\pi(\mathsf{e})\big(x^{p+1}\big)}{} =x^p\sum_{i,j}a_{i,j}x^iy^j+\sum_{i,j}a_{i,j}\frac{\alpha^pq^{jp}-1}{\alpha q^j-1}x^{p-1+i}y^j\alpha x\\
\hphantom{\pi(\mathsf{e})\big(x^{p+1}\big)}{} =\sum_{i,j}a_{i,j}\left(1+\frac{\alpha^pq^{jp}-1}{\alpha q^j-1}\alpha q^j\right)x^{p+i}y^j=\sum_{i,j}a_{i,j}\frac{\alpha^{p+1}q^{j(p+1)}-1}{\alpha q^j-1}x^{p+i}y^j.
\end{gather*}
This proves \eqref{exp} for $p\ge 0$. To cover the negative powers, we set $p>0$ and apply \eqref{ext_e} as follows
\begin{gather*}
\pi(\mathsf{e})\big(x^{-p}\big)= -x^{-p}\pi(\mathsf{e})(x^p)\pi(\mathsf{k})\big(x^p\big)^{-1}=
-x^{-p}\sum_{i,j}a_{i,j}\frac{\alpha^pq^{jp}-1}{\alpha q^j-1}x^{p-1+i}y^j\alpha^{-p}x^{-p}\\
\hphantom{\pi(\mathsf{e})\big(x^{-p}\big)}{} =-\sum_{i,j}a_{i,j}\frac{\alpha^pq^{jp}-1}{\alpha q^j-1}\alpha^{-p}q^{-jp}x^{-p-1+i}y^j=
\sum_{i,j}a_{i,j}\frac{\alpha^{-p}q^{-jp}-1}{\alpha q^j-1}x^{-p-1+i}y^j,
\end{gather*}
which finishes the proof of \eqref{exp}.

To use the relation \eqref{effe} for computing non-generic symmetries, we need certain estimates for extreme monomials of the polynomials produced by composed operators like $\pi(\mathsf{ef})$ applied to the generators~$x$ and $y$. The estimates presented in the next two lemmas, of course, do not involve (the r.h.s.\ of)~\eqref{effe} itself.

\begin{Lemma}\label{ef-fe}
In our previous notation
\begin{gather}
\pi(\mathsf{ef-fe})(x) =\sum_i\sum_{w+t=i}a_wc_tq^{(u+wr)(-v+ts)} \big(\alpha q^{is}-\alpha^{-1}\big)\big(q^{-2+iD}-1\big)x^{1+ir}y^{is}, \label{(ef-fe)x}\\
 \pi(\mathsf{ef-fe})(y) =\sum_i\sum_{w+t=i}a_wc_tq^{(u+wr)(-v+ts)} \big(\beta^{-1}q^{ir}-\beta\big) \big(1-q^{-2+iD}\big)x^{ir}y^{1+is}. \label{(ef-fe)y}
\end{gather}
\end{Lemma}

\begin{proof} Let us start with computing the polynomial $\pi(\mathsf{ef})(x)$. This involves \eqref{exy1}, \eqref{fxy1}, \eqref{def}, \eqref{def1}, and the notion of weight constants.
\begin{gather*}
\pi(\mathsf{ef})(x)=\sum_tc_t'\pi(\mathsf{e})\big(x^{-u+1+tr}y^{-v+ts}\big)\\
\hphantom{\pi(\mathsf{ef})(x)}{} =\sum_tc_t'\big(x^{-u+1+tr}\pi(\mathsf{e})\big(y^{-v+ts}\big)+\pi(\mathsf{e})\big(x^{-u+1+tr}\big)\pi(\mathsf{k}\big(y^{-v+ts}\big)\big).
\end{gather*}

One has to observe at this point that the replacement here of $\pi(\mathsf{e})\big(y^{-v+ts}\big)$ and $\pi(\mathsf{e})\big(x^{-u+1+tr}\big)$ with sums like~\eqref{exp},~\eqref{eyp} involves the coefficients $a_w'$, $a_w''$ as in \eqref{exy1}. The specific form of the latter coefficients as in \eqref{ex2}, \eqref{ey2}, respectively, implies that $a_w'$ ($a_w''$) appears to be zero when the associated denominator in~\eqref{exp} (\eqref{eyp}) is zero. This special case would require rewri\-ting~\eqref{exp}~(\eqref{eyp}) in a different form (sum of finitely many identical constants). This rewriting appears to be redundant since the corresponding terms in the sums to be substituted are totally absent, thus leaving only terms as in~\eqref{exp}~(\eqref{eyp}) with non-zero denominators.

Now we proceed with computing $\pi(\mathsf{ef})(x)$,
\begin{gather*}
\pi(\mathsf{ef})(x)=\sum_tc_t'\left(x^{-u+1+tr}\sum_wa_w''
\frac{\beta^{-v+ts}-q^{(u+wr)(-v+ts)}}{\beta-q^{(u+wr)}}x^{u+wr}y^{(w+t)s} \right.\\
 \left.\hphantom{\pi(\mathsf{ef})(x)=}{} +\sum_wa_w'\frac{\alpha^{-u+1+tr}q^{(v+ws)(-u+1+tr)}-1}{\alpha q^{v+ws}-1}x^{1+(w+t)r}y^{v+ws}\beta^{-v+ts}y^{-v+ts}\right)\\
\hphantom{\pi(\mathsf{ef})(x)}{} =\sum_w\sum_tc_t'\left(a_w'\beta^{-v+ts} \frac{\alpha^{-u+1+tr}q^{(v+ws)(-u+1+tr)}-1}{\alpha q^{v+ws}-1} \right.\\
 \left. \hphantom{\pi(\mathsf{ef})(x)=}{} +a_w''\frac{\beta^{-v+ts}-q^{(u+wr)(-v+ts)}}{\beta-q^{u+wr}}\right) x^{1+(w+t)r}y^{(w+t)s}.
\end{gather*}
Let us substitute $a_w'=a_w\big(1-\alpha q^{v+ws}\big)$, $a_w''=a_w\big(q^{u+wr}-\beta\big)$ to get
\begin{gather*}
\pi(\mathsf{ef})(x)=\sum_w\sum_ta_wc_t' \big(\beta^{-v+ts}\big(1-\alpha^{-u+1+tr}q^{(v+ws)(-u+1+tr)}\big)\\
\hphantom{\pi(\mathsf{ef})(x)=}{} +q^{(u+wr)(-v+ts)}-\beta^{-v+ts}\big)x^{1+(w+t)r}y^{(w+t)s}\\
\hphantom{\pi(\mathsf{ef})(x)}{} =\sum_w\sum_ta_wc_t'q^{(u+wr)(-v+ts)}\big(1-q^{-2+v+(w+t)D+ws}\alpha\big)x^{1+(w+t)r}y^{(w+t)s}.
\end{gather*}

A similar calculation yields
\begin{gather*}
\pi(\mathsf{fe})(x)=\sum_w\sum_ta_w'c_tq^{(u+wr)(-v+ts)} \big(\alpha^{-1}q^{-2+(w+t)D}-q^{-v+ts}\big)x^{1+(w+t)r}y^{(w+t)s}.
\end{gather*}

Finally we obtain, using $c_t'=c_t(\alpha^{-1}-q^{-v+ts})$ as in \eqref{fx2}:
\begin{gather*}
\pi(\mathsf{ef-fe})(x)=\sum_w\sum_ta_wc_tq^{(u+wr)(-v+ts)} \big(\big(\alpha^{-1}-q^{-v+ts}\big)\big(1-q^{-2+v+(w+t)D+ws}\alpha\big)\\
\hphantom{\pi(\mathsf{ef-fe})(x)=}{} -\big(1-\alpha q^{v+ws}\big)\big(\alpha^{-1}q^{-2+(w+t)D}-q^{-v+ts}\big)\big)x^{1+(w+t)r}y^{(w+t)s}\\
\hphantom{\pi(\mathsf{ef-fe})(x)}{}=\sum_w\sum_ta_wc_tq^{(u+wr)(-v+ts)} \big(q^{-2+(w+t)D}\big(\alpha q^{(w+t)s}-\alpha^{-1}\big)\\
\hphantom{\pi(\mathsf{ef-fe})(x)=}{} +\big(\alpha^{-1}-\alpha q^{(w+t)s}\big)\big)x^{1+(w+t)r}y^{(w+t)s}\\
\hphantom{\pi(\mathsf{ef-fe})(x)}{} =\sum_w\sum_ta_wc_tq^{(u+wr)(-v+ts)}\big(\alpha q^{(w+t)s}-\alpha^{-1}\big)\\
\hphantom{\pi(\mathsf{ef-fe})(x)=}{} \times \big(q^{-2+(w+t)D}-1\big)x^{1+(w+t)r}y^{(w+t)s},
\end{gather*}
which is equivalent to \eqref{(ef-fe)x}. The proof of \eqref{(ef-fe)y} goes in a similar way.
\end{proof}

\begin{Lemma}\label{extr_(ef-fe)}
\begin{gather*}
\operatorname{minind}_{1,0}(\pi(\mathsf{ef}-\mathsf{fe})(x)) \ge\operatorname{minind}_\mathsf{e}(\pi)+\operatorname{minind}_\mathsf{f}(\pi),\\
\operatorname{minind}_{0,1}(\pi(\mathsf{ef}-\mathsf{fe})(y)) \ge\operatorname{minind}_\mathsf{e}(\pi)+\operatorname{minind}_\mathsf{f}(\pi),\\
\operatorname{maxind}_{1,0}(\pi(\mathsf{ef}-\mathsf{fe})(x)) \le\operatorname{maxind}_\mathsf{e}(\pi)+\operatorname{maxind}_\mathsf{f}(\pi),\\
\operatorname{maxind}_{0,1}(\pi(\mathsf{ef}-\mathsf{fe})(y)) \le\operatorname{maxind}_\mathsf{e}(\pi)+\operatorname{maxind}_\mathsf{f}(\pi).
\end{gather*}
\end{Lemma}

\begin{proof} An obvious consequence of Lemmas \ref{ac_0} and~\ref{ef-fe}.
\end{proof}

It follows from Lemma \ref{extr_(ef-fe)} that all the non-zero monomials of the (weight) polynomials $\pi(\mathsf{ef}-\mathsf{fe})(x)$ \eqref{(ef-fe)x} and $\pi(\mathsf{ef}-\mathsf{fe})(y)$ \eqref{(ef-fe)y} should be among those with indices $i$ subject to
\begin{gather*}
\operatorname{minind}_\mathsf{e}(\pi)+
\operatorname{minind}_\mathsf{f}(\pi)\le i\le\operatorname{maxind}_\mathsf{e}(\pi)+
\operatorname{maxind}_\mathsf{f}(\pi).
\end{gather*}
In fact, this admits a further adjustment.

\begin{Lemma}\label{0ep}
Either
\begin{gather*}
\operatorname{minind}_\mathsf{e}(\pi)+\operatorname{minind}_\mathsf{f}(\pi)=0
\qquad \text{or} \qquad \operatorname{maxind}_\mathsf{e}(\pi)+\operatorname{maxind}_\mathsf{f}(\pi)=0.
\end{gather*}
\end{Lemma}

\begin{proof} To clarify the position of $0$ with respect to the interval of integers
\begin{gather}\label{inzm}
[\operatorname{minind}_\mathsf{e}(\pi)+\operatorname{minind}_\mathsf{f}(\pi);\operatorname{maxind}_\mathsf{e}(\pi)+\operatorname{maxind}_\mathsf{f}(\pi)],
\end{gather}
let us first assume that $0$ is outside of this interval. In this case the monomial corresponding to $i=0$ in both $\pi(\mathsf{ef-fe})(x)$ and $\pi(\mathsf{ef-fe})(y)$ (which is just $x$ in the r.h.s.\ of~\eqref{(ef-fe)x} and $y$ in the r.h.s.\ of~\eqref{(ef-fe)y}) is zero. This is because every term of the coefficient at this monomial involves products like $a_wc_t$ with $w+t=0$. It follows from Lemma~\ref{ac_0} that under our assumption every such product is zero. Thus, looking at~\eqref{effe}, we conclude that both weight constants $\alpha$ and $\beta$ are $1$ or $-1$. This contradicts to our assumption on $\pi$ being non-generic, which implies that we are not in the context of Lemma \ref{iz}.

Next assume that $0$ is inside the interval of integers \eqref{inzm} (not at an endpoint). Observe that the monomial in \eqref{(ef-fe)x} corresponding to $i$ being equal to one of the endpoints $\operatorname{minind}_\mathsf{e}(\pi)+\operatorname{minind}_\mathsf{f}(\pi)$ or $\operatorname{maxind}_\mathsf{e}(\pi)+\operatorname{maxind}_\mathsf{f}(\pi)$ has a constant coefficient which is just the (single) product $a_wc_tq^{(u+wr)(-v+ts)}\big(\alpha q^{is}-\alpha^{-1}\big)\big(q^{-2+iD}-1\big)$ where $w=\operatorname{minind}_\mathsf{e}(\pi)$, $t=\operatorname{minind}_\mathsf{f}(\pi)$ (respectively $w=\operatorname{maxind}_\mathsf{e}(\pi)$, $t=\operatorname{maxind}_\mathsf{f}(\pi)$). In both cases neither $a_w$ nor $c_t$ could be zero by Lemma~\ref{ecnz}. On the other hand, one deduces from~\eqref{effe} that the entire product $a_wc_tq^{(u+wr)(-v+ts)}\big(\alpha q^{is}-\alpha^{-1}\big)\big(q^{-2+iD}-1\big)$ should be zero with $i$ at each endpoint. It follows that at each endpoint one should have either $\alpha q^{is}-\alpha^{-1}=0$ or $q^{-2+iD}-1=0$.

The relation $q^{-2+iD}-1=0$, due to our assumption on $q$, is equivalent to $iD=2$.

As for $\alpha q^{is}-\alpha^{-1}=0$, one can apply additionally \eqref{abgf} to exclude $\alpha$ and obtain, under the additional assumption $s\ne 0$, that $iD=4$.

One can also reproduce the above argument with respect to \eqref{(ef-fe)y} in order to conclude that, with $i$ being at an endpoint (other than that with $iD=2$), under the additional assumption $r\ne 0$, that $iD=4$.

As it can not happen that $r=s=0$ simultaneously (this is just our present assumption on~$\pi$ being non-generic), we deduce that at this endpoint $iD=4$ with no additional assumptions.

Since $D$ is fixed, we conclude that the integer index $i$ at both endpoints of~\eqref{inzm} should have the same sign, which contradicts to our assumption on $0$ being inside the interval \eqref{inzm}. This completes the proof of lemma.
\end{proof}

\begin{Lemma}\label{iD}The extreme monomials in \eqref{(ef-fe)x} and \eqref{(ef-fe)y}, which correspond to $i\ne 0$, are subject to either $iD=4$ or $iD=2$.
\end{Lemma}

\begin{proof} By Lemma \ref{0ep} we need to prove our claim in two separate cases. Let us stick to the case $\operatorname{minind}_\mathsf{e}(\pi)+\operatorname{minind}_\mathsf{f}(\pi)=0$. The opposite case $\operatorname{maxind}_\mathsf{e}(\pi)+\operatorname{maxind}_\mathsf{f}(\pi)=0$ can be considered similarly.

So, we assume $\operatorname{minind}_\mathsf{e}(\pi)+\operatorname{minind}_\mathsf{f}(\pi) =0$, hence by the assumption of our lemma \linebreak $\operatorname{maxind}_\mathsf{e}(\pi)+\operatorname{maxind}_\mathsf{f}(\pi)>0$. The monomial in \eqref{(ef-fe)x} corresponding to $i=\operatorname{maxind}_\mathsf{e}(\pi)+ \operatorname{maxind}_\mathsf{f}(\pi)$ has a~constant coefficient which is just the (single) product $a_wc_tq^{(u+wr)(-v+ts)}\big(\alpha q^{is}-\alpha^{-1}\big)\big(q^{-2+iD}-1\big)$ where $w=\operatorname{maxind}_\mathsf{e}(\pi)$, $t=\operatorname{maxind}_\mathsf{f}(\pi)$. Neither $a_w$ nor $c_t$ in the latter product could be zero by Lemma~\ref{ecnz}. On the other hand, one deduces from \eqref{effe} that the entire product $a_wc_tq^{(u+wr)(-v+ts)}\big(\alpha q^{is}-\alpha^{-1}\big)\big(q^{-2+iD}-1\big)$ should be zero since $i>0$ (while $(r,s)\ne(0,0)$, as we are in type~(IV)). It follows that one should have either $\alpha q^{is}-\alpha^{-1}=0$ or $q^{-2+iD}-1=0$.

The relation $q^{-2+iD}-1=0$, due to our assumption on $q$, is equivalent to $iD=2$.

As for $\alpha q^{is}-\alpha^{-1}=0$, one can apply additionally \eqref{abgf} to exclude $\alpha$ and obtain, under the additional assumption $s\ne 0$, that $iD=4$.

One can also reproduce the above argument with respect to~\eqref{(ef-fe)y} in order to conclude that, with $i$ other than that with $iD=2$, under the additional assumption $r\ne 0$, that $iD=4$.

As it can not happen that $r=s=0$ simultaneously (since $\pi$ is type~(IV)), we deduce that $iD=4$ with no additional assumptions.
\end{proof}

\begin{Corollary}\label{5terms}The sums in \eqref{ex2}--\eqref{fy2} can have at most 5 (non-zero) terms.
\end{Corollary}

\begin{Remark}In fact, the number of terms in the sums in \eqref{ex2}--\eqref{fy2} appears to be even lower, namely at most 4. This will become clear below, after final computing of the coefficients.
\end{Remark}

In view of Lemmas \ref{ac_0}, \ref{ecnz}, \ref{0ep}, and Corollary \ref{5terms}, we can now arrange one more adjustment of \eqref{ex2}--\eqref{fy2}. For that, let us introduce three more parameters $M,L,N\in\mathbb{Z}$. Looking at Lemma \ref{0ep}, we consider first the case $\operatorname{minind}_\mathsf{e}(\pi)+\operatorname{minind}_\mathsf{f}(\pi)=0$. Then with $0\le L\le\operatorname{maxind}_\mathsf{e}(\pi)+
\operatorname{maxind}_\mathsf{f}(\pi)=N$, one can rewrite \eqref{ex2}--\eqref{fy2} as follows
\begin{gather}
\pi(\mathsf{e})(x) =\sum_{w=0}^La_{M+w}\big(1-\alpha q^{v+(M+w)s}\big)x^{u+1+(M+w)r}y^{v+(M+w)s}, \label{ex3}\\
\pi(\mathsf{e})(y) =\sum_{w=0}^La_{M+w}\big(q^{u+(M+w)r}-\beta\big)x^{u+(M+w)r}y^{v+1+(M+w)s},\label{ey3}\\
\pi(\mathsf{f})(x) =\sum_{t=0}^{N-L}c_{-M+t}\big(\alpha^{-1}-q^{-v+(-M+t)s}\big)x^{-u+1+(-M+t)r}y^{-v+(-M+t)s}, \label{fx3}\\
\pi(\mathsf{f})(y)=\sum_{t=0}^{N-L}c_{-M+t}\big(\beta^{-1}q^{-u+(-M+t)r}-1\big)x^{-u+(-M+t)r}y^{-v+1+(-M+t)s}. \label{fy3}
\end{gather}
Note that by Lemma \ref{iD} $|N|\in\{1,2,4\}$, and there must be certain dependencies between $a_{M+w}$ and $c_{-M+t}$.

In the case $\operatorname{maxind}_\mathsf{e}(\pi)+\operatorname{maxind}_\mathsf{f}(\pi)=0$, with $N=\operatorname{minind}_\mathsf{e}(\pi)+\operatorname{minind}_\mathsf{f}(\pi)\le L\le 0$, one has
\begin{gather}
\pi(\mathsf{e})(x) =\sum_{w=L}^0a_{M+w}\big(1-\alpha q^{v+(M+w)s}\big)x^{u+1+(M+w)r}y^{v+(M+w)s},\label{ex4}\\
\pi(\mathsf{e})(y) =\sum_{w=L}^0a_{M+w}\big(q^{u+(M+w)r}-\beta\big)x^{u+(M+w)r}y^{v+1+(M+w)s}, \label{ey4}\\
\pi(\mathsf{f})(x) =\sum_{t=N-L}^0c_{-M+t}\big(\alpha^{-1}-q^{-v+(-M+t)s}\big)x^{-u+1+(-M+t)r}y^{-v+(-M+t)s}, \label{fx4}\\
\pi(\mathsf{f})(y) =\sum_{t=N-L}^0c_{-M+t}\big(\beta^{-1}q^{-u+(-M+t)r}-1\big)x^{-u+(-M+t)r}y^{-v+1+(-M+t)s}. \label{fy4}
\end{gather}

\begin{Lemma}\label{param_change}
The collection of $($non-generic$)$ symmetries corresponding to $\operatorname{maxind}_\mathsf{e}(\pi)+$ \linebreak $\operatorname{maxind}_\mathsf{f}(\pi)=0$ coincides with the collection of symmetries with $\operatorname{minind}_\mathsf{e}(\pi)+\operatorname{minind}_\mathsf{f}(\pi)=0$. More precisely, each symmetry from the first collection becomes a symmetry from the second collection after a suitable change of parameters, and vice versa.
\end{Lemma}

\begin{proof} The symmetries with $\operatorname{maxind}_\mathsf{e}(\pi)+ \operatorname{maxind}_\mathsf{f}(\pi)=0$ are described by \eqref{ex4}--\eqref{fy4}, together with $\pi(\mathsf{k})(x)=\alpha x$, $\pi(\mathsf{k})(y)=\beta y$, where the pair of weight constants $\binom{\alpha}{\beta}$ is a (minimal) solution of~\eqref{wce}. This description is certainly modulo some dependencies between $a_{M+w}$ and $c_{-M+t}$.

Let us substitute $M=-M'$, $r=-r'$, $s=-s'$, $w=-w'$, $t=-t'$. The latter two changes of indices in sums forces also change of the parameters $N$ and $L$. Now \eqref{ex4}--\eqref{fy4} acquire the form
\begin{gather*}
\pi(\mathsf{e})(x) =\sum_{w'=0}^{-L}a_{-M'-w'}\big(1-\alpha q^{v+(M'+w')s'}\big)x^{u+1+(M'+w')r'}y^{v+(M'+w')s'},\\
\pi(\mathsf{e})(y) =\sum_{w'=0}^{-L}a_{-M'-w'}\big(q^{u+(M'+w')r'}-\beta\big)x^{u+(M'+w')r'}y^{v+1+(M'+w')s'},\\
\pi(\mathsf{f})(x) =\sum_{t'=0}^{-N+L}c_{M'-t'}\big(\alpha^{-1}-q^{-v+(-M'+t')s'}\big)x^{-u+1+(-M'+t')r'}y^{-v+(-M'+t')s'},\\
\pi(\mathsf{f})(y) =\sum_{t'=0}^{-N+L}c_{M'-t'}\big(\beta^{-1}q^{-u+(-M'+t')r'}-1\big)x^{-u+(-M'+t')r'}y^{-v+1+(-M'+t')s'}.
\end{gather*}
Clearly, this is the same as \eqref{ex3}--\eqref{fy3} up to replacement of (in fact, reindexing) the parame\-ters~$a_i$,~$c_j$. The set of minimal solutions of~\eqref{wce} (weight constants) obviously remains intact under the above reindexing, hence the operator $\pi(\mathsf{k})$ remains the same. Another obvious observation is that this reindexing also remains intact the elements $\pi(\mathsf{e})(x)$, $\pi(\mathsf{e})(y)$, $\pi(\mathsf{f})(x)$, $\pi(\mathsf{f})(y)$ of $\mathbb{C}_q\big[x^{\pm 1},y^{\pm 1}\big]$ under a fixed set of~$a_i$,~$c_j$ before reindexing, since each of those complex constants remains the same after reindexing. It follows that the operators $\pi(\mathsf{e})$, $\pi(\mathsf{f})$ remain the same after the above reindexing, as this is first an application of the coproduct in $U_q(\mathfrak{sl}_2)$ (via using the Sweedler sigma-notation) while extending to monomials, and then extending by linearity to the entire $\mathbb{C}_q\big[x^{\pm 1},y^{\pm 1}\big]$. Both procedures are independent of the parameters of symmetries.

Thus, the symmetry (under a fixed set of complex parameters) appears to be the same after reindexing, but now it possesses the property $\operatorname{minind}_\mathsf{e}(\pi)+
\operatorname{minind}_\mathsf{f}(\pi)=0$ with respect to the new integral parameters.
\end{proof}

\begin{Remark}\label{rest_form}Lemma \ref{param_change} implies that, in writing down a complete list of non-generic $U_q(\mathfrak{sl}_2)$-symmetries on $\mathbb{C}_q\big[x^{\pm 1},y^{\pm 1}\big]$, it suffices to restrict oneself to the symmetries with $\operatorname{minind}_\mathsf{e}(\pi)+\operatorname{minind}_\mathsf{f}(\pi)=0$.
\end{Remark}

\begin{Remark}
In the context of Lemma \ref{param_change} and Remark \ref{rest_form}, it is useful to reindex the (non-zero) complex constants $a_i$, $c_j$ in such a way that $a_{M+w}$ as in \eqref{ex3}--\eqref{ey3} becomes $a_w$ and $c_{-M+t}$ as in \eqref{fx3}--\eqref{fy3} becomes~$c_t$. This allows one to simplify the notation and to rewrite \eqref{ex3}--\eqref{fy3} and \eqref{(ef-fe)x}--\eqref{(ef-fe)y} in the form as follows, suitable for subsequent applications:
\begin{gather}
\pi(\mathsf{e})(x) =\sum_{w=0}^La_w\big(1-\alpha q^{v+(M+w)s}\big)x^{u+1+(M+w)r}y^{v+(M+w)s},\label{ex6}\\
 \pi(\mathsf{e})(y) =\sum_{w=0}^La_w\big(q^{u+(M+w)r}-\beta\big)x^{u+(M+w)r}y^{v+1+(M+w)s},\label{ey6}\\
\pi(\mathsf{f})(x) =\sum_{t=0}^{N-L}c_t\big(\alpha^{-1}-q^{-v+(-M+t)s}\big)x^{-u+1+(-M+t)r}y^{-v+(-M+t)s}, \label{fx6}\\
\pi(\mathsf{f})(y) =\sum_{t=0}^{N-L}c_t\big(\beta^{-1}q^{-u+(-M+t)r}-1\big)x^{-u+(-M+t)r}y^{-v+1+(-M+t)s}, \label{fy6}\\
\pi(\mathsf{ef-fe})(x) =\sum_i\sum_{\substack{0\le w\le L\\ 0\le t\le N-L\\ w+t=i}}a_wc_tq^{(u+(M+w)r)(-v+(-M+t)s)}\big(\alpha q^{is}-\alpha^{-1}\big)\nonumber\\
\hphantom{\pi(\mathsf{ef-fe})(x) =}{}\times \big(q^{-2+iD}-1\big) x^{1+ir}y^{is}, \label{(ef-fe)x2}\\
 \pi(\mathsf{ef-fe})(y) =\sum_i\sum_{\substack{0\le w\le L\\ 0\le t\le N-L\\ w+t=i}}a_wc_tq^{(u+(M+w)r)(-v+(-M+t)s)}\big(\beta^{-1}q^{ir}-\beta\big)\nonumber\\
 \hphantom{\pi(\mathsf{ef-fe})(y) =}{}\times \big(1-q^{-2+iD}\big) x^{ir}y^{1+is}. \label{(ef-fe)y2}
\end{gather}
\end{Remark}

\section{A complete list of non-generic symmetries}\label{clngs}

The discussion of Section \ref{wpem} indicates that non-generic symmetries are to be searched for in the form \eqref{ex6}--\eqref{fy6}, together with
\begin{gather*}\pi(\mathsf{k})(x)=\alpha x,\qquad\pi(\mathsf{k})(y)=\beta y,\end{gather*}
where the pair of weight constants $\binom{\alpha}{\beta}$ is such a solution of \eqref{wce} that the pair $(r,s)$ is minimal with respect to $\binom{\alpha}{\beta}$.

Now we are in a position to write down an explicit form of non-generic $U_q(\mathfrak{sl}_2)$-symmetries on $\mathbb{C}_q\big[x^{\pm 1},y^{\pm 1}\big]$. The last step in doing this is in computing relations between (and, possibly, excluding some of) the constant coefficients $a_w$, $c_t$ at monomials in \eqref{ex6}--\eqref{fy6}. This is to be done (after explicit calculation of weight constants) via an application of~\eqref{effe} to \eqref{(ef-fe)x2}, \eqref{(ef-fe)y2}, where all the monomials should be zero except those with $i=0$; the latter must have coefficients deduced from~\eqref{effe}.

We keep the notation of Section~\ref{wpem} concerning the integral parameters $r$, $s$, $u$, $v$, $D$, $N$, $L$ of non-generic symmetries. It follows from Proposition~\ref{D_inv} and Lemma~\ref{iD} that we have to distinguish the 3 classes of symmetries with $|D|=1,2,\text{\ or\ }4$. Also, in view of Lemma~\ref{param_change}, we restrict our considerations to the case $\operatorname{minind}_\mathsf{e}(\pi)+\operatorname{minind}_\mathsf{f}(\pi)=0$. It follows that $\operatorname{maxind}_\mathsf{e}(\pi)+ \operatorname{maxind}_\mathsf{f}(\pi)>0$, hence, in view of Lemma~\ref{iD}, $D=1,2,\text{\ or\ }4$. However, it will become clear below that the explicit form of relationship between the coefficients $a_w$, $c_t$ depends on $L$, so we need to partition the class of symmetries into finer series corresponding to specific values of $L$. In fact, the values of $L$ involved here determine the number of terms (utmost, with all $a_w$, $c_t$ non-zero, $0\le w\le L$, $0\le t\le N-L$) at $\pi(\mathsf{e})(x)$, $\pi(\mathsf{e})(y)$, $\pi(\mathsf{f})(x)$, $\pi(\mathsf{f})(y)$.

The names of series of symmetries to be used below are of the form $D_iG_jE_{L+1}F_{N-L+1}$, where~$i$ is the value of $D$ within the series (see also Proposition~\ref{D_inv} and the subsequent remark), $j$ is the value of the invariant $G=\operatorname{gcd}(r,s)$, $L+1$ the utmost number of terms at $\pi(\mathsf{e})(x)$, $\pi(\mathsf{e})(y)$, $N-L+1$ the utmost number of terms at $\pi(\mathsf{f})(x)$, $\pi(\mathsf{f})(y)$.

In what follows we restrict ourselves to writing down the final form of the series of non-generic $U_q(\mathfrak{sl}_2)$-symmetries on $\mathbb{C}_q\big[x^{\pm 1},y^{\pm 1}\big]$. The related calculation of the coefficients is completely routine and thus omitted.

\subsection[The case $D=1$]{The case $\boldsymbol{D=1}$}

With $D=1$, the weight constants $\alpha$, $\beta$ are determined by~\eqref{abgf} unambiguously, once the integral parameters $r$, $s$, $u$, $v$ are given. The corresponding values for the weight constants are assumed to be substituted to \eqref{ex6}--\eqref{(ef-fe)y2} prior to final calculations. The minimality condition for~$(r,s)$ here is an immediate consequence of $D=rv-su=1$, as~$r$,~$s$ are coprime. Hence in the present case $G=\operatorname{gcd}(r,s)=1$.

\subsubsection*{$\boldsymbol{D_1G_1E_1F_3}$}
\begin{gather*}
\pi(\mathsf{k})(x)= q^{-2s}x,\qquad\pi(\mathsf{k})(y)=q^{2r}y, \\
\pi(\mathsf{e})(x)=-c_0^{-1}q^{(u+Mr)(v+Ms)+3}\big(1-q^2\big)^{-2}\big(1-q^{v+(M-2)s}\big)x^{u+1+Mr}y^{v+Ms},\\
\pi(\mathsf{e})(y)=-c_0^{-1}q^{(u+Mr)(v+Ms)+3}\big(1-q^2\big)^{-2}\big(q^{u+Mr}-q^{2r}\big)x^{u+Mr}y^{v+1+Ms}, \\
\pi(\mathsf{f})(x)=c_0\big(q^{2s}-q^{-v-Ms}\big)x^{-u+1-Mr}y^{-v-Ms}\\
\hphantom{\pi(\mathsf{f})(x)=}{}+ c_2\big(q^{2s}-q^{-v+(-M+2)s}\big)x^{-u+1+(-M+2)r}y^{-v+(-M+2)s}\\
\hphantom{\pi(\mathsf{f})(x)=}{}+ c_4\big(q^{2s}-q^{-v+(-M+4)s}\big)x^{-u+1+(-M+4)r}y^{-v+(-M+4)s}, \\
\pi(\mathsf{f})(y)=c_0\big(q^{-u+(-M-2)r}-1\big)x^{-u-Mr}y^{-v+1-Ms}\\
\hphantom{\pi(\mathsf{f})(y)=}{}+ c_2\big(q^{-u-Mr}-1\big)x^{-u+(-M+2)r}y^{-v+1+(-M+2)s}\\
\hphantom{\pi(\mathsf{f})(y)=}{}+c_4\big(q^{-u+(-M+2)r}-1\big)x^{-u+(-M+4)r}y^{-v+1+(-M+4)s}.
\end{gather*}
Here $r,s,u,v,M\in\mathbb{Z}$, $D=rv-su=1$, $G=\operatorname{gcd}(r,s)=1$, $c_0,c_2,c_4\in\mathbb{C}$, $c_0\ne 0$. The name of series is due to the {\it real} number of terms, unlike $D_1G_1E_1F_5$ (as it might be in correspondence with the value $L=0$). This is because the calculation of coefficients in \eqref{fx6}, \eqref{fy6} with $L=0$ yields $c_1=c_3=0$.

\subsubsection*{$\boldsymbol{D_1G_1E_2F_4}$}
\begin{gather*}
\pi(\mathsf{k})(x)= q^{-2s}x ,\qquad\pi(\mathsf{k})(y)=q^{2r}y ,\\
\pi(\mathsf{e})(x)= -c_0^{-1}q^{(u+Mr)(v+Ms)+3}\big(1-q^2\big)^{-2}\big(1-q^{v+(M-2)s}\big)x^{u+1+Mr}y^{v+Ms}\\
\hphantom{\pi(\mathsf{e})(x)=}{}+ c_0^{-2}c_1q^{(u+Mr)(v+Ms)+2su+4+2Mrs}\big(1-q^2\big)^{-2}\big(1-q^{v+(M-1)s}\big)\\
\hphantom{\pi(\mathsf{e})(x)=}{}\times x^{u+1+(M+1)r}y^{v+(M+1)Ms} ,\\
\pi(\mathsf{e})(y)= -c_0^{-1}q^{(u+Mr)(v+Ms)+3}\big(1-q^2\big)^{-2}\big(q^{u+Mr}-q^{2r}\big)x^{u+1+(M+1)r}y^{v+(M+1)s}\\
\hphantom{\pi(\mathsf{e})(y)=}{}+ c_0^{-2}c_1q^{(u+Mr)(v+Ms)+2su+4+2Mrs}\big(1-q^2\big)^{-2}\big(q^{u+(M+1)r}-q^{2r}\big)\\
\hphantom{\pi(\mathsf{e})(y)=}{}\times x^{u+(M+1)r}y^{v+1+(M+1)Ms} ,\\
\pi(\mathsf{f})(x)= c_0\big(q^{2s}-q^{-v-Ms}\big)x^{-u+1-Mr}y^{-v-Ms}\\
\hphantom{\pi(\mathsf{f})(x)=}{}+ c_1\big(q^{2s}-q^{-v+(-M+1)s}\big)x^{-u+1+(-M+1)r}y^{-v+(-M+1)s}\\
\hphantom{\pi(\mathsf{f})(x)=}{}+ c_2\big(q^{2s}-q^{-v+(-M+2)s}\big)x^{-u+1+(-M+2)r}y^{-v+(-M+2)s}\\
\hphantom{\pi(\mathsf{f})(x)=}{}+ c_0^{-1}c_1c_2q^{4Mrs-Mrv-M^2rs-Mr+2rs}\big(q^{2s}-q^{-v+(-M+3)s}\big)\\
\hphantom{\pi(\mathsf{f})(x)=}{}\times x^{-u+1+(-M+3)r}y^{-v+(-M+3)s} ,\\
\pi(\mathsf{f})(y)= c_0\big(q^{-u+(-M-2)r}-1\big)x^{-u-Mr}y^{-v+1-Ms}\\
\hphantom{\pi(\mathsf{f})(y)=}{} + c_1\big(q^{-u+(-M-1)r}-1\big)x^{-u+(-M+1)r}y^{-v+1+(-M+1)s}\\
\hphantom{\pi(\mathsf{f})(y)=}{} + c_2\big(q^{-u-Mr}-1\big)x^{-u+(-M+2)r}y^{-v+1+(-M+2)s}\\
\hphantom{\pi(\mathsf{f})(y)=}{} + c_0^{-1}c_1c_2q^{4Mrs-Mrv-M^2rs-Mr+2rs}\big(q^{-u+(-M+1)r}-1\big)\\
\hphantom{\pi(\mathsf{f})(y)=}{}\times x^{-u+(-M+3)r}y^{-v+1+(-M+3)s} .
\end{gather*}
Here $r,s,u,v,M\in\mathbb{Z}$, $D=rv-su=1$, $G=\operatorname{gcd}(r,s)=1$, $c_0,c_1,c_2\in\mathbb{C}$, $c_0\ne 0$. This corresponds to $L=1$.

\subsubsection*{$\boldsymbol{D_1G_1E_3F_3}$}
\begin{gather*}
\pi(\mathsf{k})(x)= q^{-2s}x ,\qquad\pi(\mathsf{k})(y)=q^{2r}y ,\\
\pi(\mathsf{e})(x)= a_0\big(1-q^{v+(M-2)s}\big)x^{u+1+Mr}y^{v+Ms}+ a_1\big(1-q^{v+(M-1)s}\big)x^{u+1+(M+1)r}y^{v+(M+1)Ms}\\
\hphantom{\pi(\mathsf{e})(x)=}{} + a_2\big(1-q^{v+Ms}\big)x^{u+1+(M+2)r}y^{v+(M+2)s} ,\\
\pi(\mathsf{e})(y)= a_0\big(q^{u+Mr}-q^{2r}\big)x^{u+Mr}y^{v+1+Ms}+ a_1\big(q^{u+(M+1)r}-q^{2r}\big)x^{u+(M+1)r}y^{v+1+(M+1)Ms}\\
\hphantom{\pi(\mathsf{e})(y)=}{}+ a_2\big(q^{u+(M+2)r}-q^{2r}\big)x^{u+(M+2)r}y^{v+1+(M+2)s} ,\\
\pi(\mathsf{f})(x)= -a_0^{-1}q^{(u+Mr)(v+Ms)+3}\big(1-q^2\big)^{-2}\big(q^{2s}-q^{-v-Ms}\big)x^{-u+1-Mr}y^{-v-Ms}\\
\hphantom{\pi(\mathsf{f})(x)=}{}+ a_0^{-2}a_1q^{(u+Mr)(v+Ms)+2-2su-2Mrs}\big(1-q^2\big)^{-2}\big(q^{2s}-q^{-v+(-M+1)s}\big)\\
\hphantom{\pi(\mathsf{f})(x)=}{}\times x^{-u+1+(-M+1)r}y^{-v+(-M+1)s}\\
\hphantom{\pi(\mathsf{f})(x)=}{} -a_0^{-2}a_2q^{(u+Mr)(v+Ms)+1-4su-4Mrs}\big(1-q^2\big)^{-2}\big(q^{2s}-q^{-v+(-M+2)s}\big)\\
\hphantom{\pi(\mathsf{f})(x)=}{}\times x^{-u+1+(-M+2)r}y^{-v+(-M+2)s} ,\\
\pi(\mathsf{f})(y)= -a_0^{-1}q^{(u+Mr)(v+Ms)+3}\big(1-q^2\big)^{-2}\big(q^{-u+(-M-2)r}-1\big)x^{-u-Mr}y^{-v+1-Ms}\\
\hphantom{\pi(\mathsf{f})(y)=}{}+ a_0^{-2}a_1q^{(u+Mr)(v+Ms)+2-2su-2Mrs}\big(1-q^2\big)^{-2}\big(q^{-u+(-M-1)r}-1\big)\\
\hphantom{\pi(\mathsf{f})(y)=}{}\times x^{-u+(-M+1)r}y^{-v+1+(-M+1)s}\\
\hphantom{\pi(\mathsf{f})(y)=}{} -a_0^{-2}a_2q^{(u+Mr)(v+Ms)+1-4su-4Mrs}\big(1-q^2\big)^{-2}\big(q^{-u-Mr}-1\big)\\
\hphantom{\pi(\mathsf{f})(y)=}{}\times x^{-u+(-M+2)r}y^{-v+1+(-M+2)s} .
\end{gather*}
Here $r,s,u,v,M\in\mathbb{Z}$, $D=rv-su=1$, $G=\operatorname{gcd}(r,s)=1$, $a_0,a_1,a_2\in\mathbb{C}$, $a_0,a_1\ne 0$. This corresponds to $L=2$.

In the case $a_1=0$ the last terms at $\pi(\mathsf{f})(x)$, $\pi(\mathsf{f})(y)$ appear to be different and we obtain the series

\subsubsection*{$\boldsymbol{D_1G_1E_2F_2}$}
\begin{gather*}
\pi(\mathsf{k})(x)= q^{-2s}x ,\qquad\pi(\mathsf{k})(y)=q^{2r}y ,\\
\pi(\mathsf{e})(x)= a_0\big(1-q^{v+(M-2)s}\big)x^{u+1+Mr}y^{v+Ms}+ a_2\big(1-q^{v+Ms}\big)x^{u+1+(M+2)r}y^{v+(M+2)s} ,\\
\pi(\mathsf{e})(y)= a_0\big(q^{u+Mr}-q^{2r}\big)x^{u+Mr}y^{v+1+Ms}+ a_2\big(q^{u+(M+2)r}-q^{2r}\big)x^{u+(M+2)r}y^{v+1+(M+2)s} ,\\
\pi(\mathsf{f})(x)= -a_0^{-1}q^{(u+Mr)(v+Ms)+3}\big(1-q^2\big)^{-2}\big(q^{2s}-q^{-v-Ms}\big)x^{-u+1-Mr}y^{-v-Ms}\\
\hphantom{\pi(\mathsf{f})(x)=}{}+ c_2\big(q^{2s}-q^{-v+(-M+2)s}\big)x^{-u+1+(-M+2)r}y^{-v+(-M+2)s} ,\\
\pi(\mathsf{f})(y)= -a_0^{-1}q^{(u+Mr)(v+Ms)+3}\big(1-q^2\big)^{-2}\big(q^{-u+(-M-2)r}-1\big)x^{-u-Mr}y^{-v+1-Ms}\\
\hphantom{\pi(\mathsf{f})(y)=}{}+ c_2\big(q^{-u-Mr}-1\big)x^{-u+(-M+2)r}y^{-v+1+(-M+2)s} .
\end{gather*}
Here $r,s,u,v,M\in\mathbb{Z}$, $D=rv-su=1$, $G=\operatorname{gcd}(r,s)=1$, $a_0,a_2,c_2\in\mathbb{C}$, $a_0\ne 0$. This is just the case $L=2$, $a_1=0$ in~\eqref{ex6},~\eqref{ey6}, hence the name of series corresponding to the {\it real} number of terms.

\subsubsection*{$\boldsymbol{D_1G_1E_4F_2}$}
\begin{gather*}
\pi(\mathsf{k})(x)= q^{-2s}x ,\qquad\pi(\mathsf{k})(y)=q^{2r}y ,\\
\pi(\mathsf{e})(x)= a_0\big(1-q^{v+(M-2)s}\big)x^{u+1+Mr}y^{v+Ms}+ a_1\big(1-q^{v+(M-1)s}\big)x^{u+1+(M+1)r}y^{v+(M+1)Ms}\\
\hphantom{\pi(\mathsf{e})(x)=}{}+ a_2\big(1-q^{v+Ms}\big)x^{u+1+(M+2)r}y^{v+(M+2)s}\\
\hphantom{\pi(\mathsf{e})(x)=}{}+ a_0^{-1}a_1a_2q^{2rs}\big(1-q^{v+(M+1)s}\big)x^{u+1+(M+3)r}y^{v+(M+3)s} ,\\
\pi(\mathsf{e})(y)= a_0\big(q^{u+Mr}-q^{2r}\big)x^{u+Mr}y^{v+1+Ms}+ a_1\big(q^{u+(M+1)r}-q^{2r}\big)x^{u+(M+1)r}y^{v+1+(M+1)s}\\
\hphantom{\pi(\mathsf{e})(y)=}{}+ a_2\big(q^{u+(M+2)r}-q^{2r}\big)x^{u+(M+2)r}y^{v+1+(M+2)s}\\
\hphantom{\pi(\mathsf{e})(y)=}{}+ a_0^{-1}a_1a_2q^{2rs}\big(q^{u+(M+3)r}-1\big)x^{u+(M+3)r}y^{v+1+(M+3)s} ,\\
\pi(\mathsf{f})(x)= -a_0^{-1}q^{(u+Mr)(v+Ms)+3}\big(1-q^2\big)^{-2}\big(q^{2s}-q^{-v-Ms}\big)x^{-u+1-Mr}y^{-v-Ms}\\
\hphantom{\pi(\mathsf{f})(x)=}{} + a_0^{-2}a_1q^{(u+Mr)(v+Ms)+2-2su-2Mrs}\big(1-q^2\big)^{-2} \big(q^{2s}-q^{-v+(-M+1)s}\big)\\
\hphantom{\pi(\mathsf{f})(x)=}{}\times x^{-u+1+(-M+1)r}y^{-v+(-M+1)s} ,\\
\pi(\mathsf{f})(y)= -a_0^{-1}q^{(u+Mr)(v+Ms)+3}\big(1-q^2\big)^{-2}\big(q^{-u+(-M-2)r}-1\big)x^{-u-Mr}y^{-v+1-Ms}\\
\hphantom{\pi(\mathsf{f})(y)=}{}+ a_0^{-2}a_1q^{(u+Mr)(v+Ms)+2-2su-2Mrs}\big(1-q^2\big)^{-2}\big(q^{-u+(-M-1)r}-1\big)\\
\hphantom{\pi(\mathsf{f})(y)=}{}\times x^{-u+(-M+1)r}y^{-v+1+(-M+1)s} .
\end{gather*}
Here $r,s,u,v,M\in\mathbb{Z}$, $D=rv-su=1$, $G=\operatorname{gcd}(r,s)=1$, $a_0,a_1,a_2\in\mathbb{C}$, $a_0\ne 0$. This corresponds to $L=3$.

\subsubsection*{$\boldsymbol{D_1G_1E_3F_1}$}
\begin{gather*}
\pi(\mathsf{k})(x)= q^{-2s}x ,\qquad\pi(\mathsf{k})(y)=q^{2r}y ,\\
 \pi(\mathsf{e})(x)= a_0\big(1-q^{v+(M-2)s}\big)x^{u+1+Mr}y^{v+Ms}+ a_2\big(1-q^{v+Ms}\big)x^{u+1+(M+2)r}y^{v+(M+2)s}\\
 \hphantom{\pi(\mathsf{e})(x)=}{}+ a_4\big(1-q^{v+(M+2)s}\big)x^{u+1+(M+4)r}y^{v+(M+4)s} ,\\
\pi(\mathsf{e})(y)= a_0\big(q^{u+Mr}-q^{2r}\big)x^{u+Mr}y^{v+1+Ms}+ a_2\big(q^{u+(M+2)r}-q^{2r}\big)x^{u+(M+2)r}y^{v+1+(M+2)s}\\
\hphantom{\pi(\mathsf{e})(y)=}{}+ a_4\big(q^{u+(M+4)r}-q^{2r}\big)x^{u+(M+4)r}y^{v+1+(M+4)s} ,\\
\pi(\mathsf{f})(x)= a_0^{-1}q^{(u+Mr)(v+Ms)+3}\big(1-q^2\big)^{-2}\big(q^{2s}-q^{-v-Ms}\big)x^{-u+1-Mr}y^{-v-Ms} ,\\
\pi(\mathsf{f})(y)= a_0^{-1}q^{(u+Mr)(v+Ms)+3}\big(1-q^2\big)^{-2}\big(q^{-u+(-M-2)r}-1\big)x^{-u-Mr}y^{-v+1-Ms} .
\end{gather*}
Here $r,s,u,v,M\in\mathbb{Z}$, $D=rv-su=1$, $G=\operatorname{gcd}(r,s)=1$, $a_0,a_2,a_4\in\mathbb{C}$, $a_0\ne 0$. The name of series is due to the {\it real} number of terms, unlike $D_1G_1E_5F_1$ (as it might be in correspondence with the value $L=4$). This is because the calculation of coefficients in \eqref{ex6}, \eqref{ey6} with $L=4$ yields $a_1=a_3=0$.

\subsubsection*{The case $\boldsymbol{D=1}$, $\boldsymbol{N=2}$}

Under the assumption $D=1$, by Lemma~\ref{iD} we have also to consider the case $N=2$. It turns out that this way we find no additional symmetries. To verify this, we apply the above procedure of computing coefficients just to write down the corresponding series. We refrain from setting names to these series, because they are all embeddable to the above series with $N=4$.
\begin{gather*}
\mathbf{L}=\mathbf{0}\\
\pi(\mathsf{k})(x)= q^{-2s}x ,\qquad\pi(\mathsf{k})(y)=q^{2r}y,\\
\pi(\mathsf{e})(x)= -c_0^{-1}q^{(u+Mr)(v+Ms)+3}\big(1-q^2\big)^{-2}\big(1-q^{v+(M-2)s}\big)x^{u+1+Mr}y^{v+Ms},\\
\pi(\mathsf{e})(y)= -c_0^{-1}q^{(u+Mr)(v+Ms)+3}\big(1-q^2\big)^{-2}\big(q^{u+Mr}-q^{2r}\big)x^{u+Mr}y^{v+1+Ms},\\
\pi(\mathsf{f})(x)= c_0\big(q^{2s}-q^{-v-Ms}\big)x^{-u+1-Mr}y^{-v-Ms}\\
\hphantom{\pi(\mathsf{f})(x)=}{} + c_2\big(q^{2s}-q^{-v+(-M+2)s}\big)x^{-u+1+(-M+2)r}y^{-v+(-M+2)s},\\
\pi(\mathsf{f})(y)= c_0\big(q^{-u+(-M-2)r}-1\big)x^{-u-Mr}y^{-v+1-Ms}\\
\hphantom{\pi(\mathsf{f})(y)=}{}+ c_2\big(q^{-u-Mr}-1\big)x^{-u+(-M+2)r}y^{-v+1+(-M+2)s}.
\end{gather*}
Here $r,s,u,v,M\in\mathbb{Z}$, $D=rv-su=1$, $G=\operatorname{gcd}(r,s)=1$, $c_0,c_2\in\mathbb{C}$, $c_0\ne 0$. This is embeddable into $D_1G_1E_1F_3$ by setting there $c_4=0$.
\begin{gather*}
 \mathbf{L}=\mathbf{1}\\
 \pi(\mathsf{k})(x)= q^{-2s}x ,\qquad\pi(\mathsf{k})(y)=q^{2r}y ,\\
\pi(\mathsf{e})(x)= a_0\big(1-q^{v+(M-2)s}\big)x^{u+1+Mr}y^{v+Ms}\!+ a_1\big(1-q^{v+(M-1)s}\big)x^{u+1+(M+1)r}y^{v+(M+1)Ms} ,\\
\pi(\mathsf{e})(y)= a_0\big(q^{u+Mr}-q^{2r}\big)x^{u+Mr}y^{v+1+Ms}\!+ a_1\big(q^{u+(M+1)r}-q^{2r}\big)x^{u+(M+1)r}y^{v+1+(M+1)Ms} ,\\
\pi(\mathsf{f})(x)= -a_0^{-1}q^{(u+Mr)(v+Ms)+3}\big(1-q^2\big)^{-2}\big(q^{2s}-q^{-v-Ms}\big)x^{-u+1-Mr}y^{-v-Ms}\\
\hphantom{\pi(\mathsf{f})(x)=}{}+ a_0^{-2}a_1q^{(u+Mr)(v+Ms)+2-2su-2Mrs}\big(1-q^2\big)^{-2}\big(q^{2s}-q^{-v+(-M+1)s}\big)\\
\hphantom{\pi(\mathsf{f})(x)=}{}\times x^{-u+1+(-M+1)r}y^{-v+(-M+1)s} ,\\
\pi(\mathsf{f})(y)= -a_0^{-1}q^{(u+Mr)(v+Ms)+3}\big(1-q^2\big)^{-2}\big(q^{-u+(-M-2)r}-1\big)x^{-u-Mr}y^{-v+1-Ms}\\
\hphantom{\pi(\mathsf{f})(y)=}{}+ a_0^{-2}a_1q^{(u+Mr)(v+Ms)+2-2su-2Mrs}\big(1-q^2\big)^{-2}\big(q^{-u+(-M-1)r}-1\big)\\
\hphantom{\pi(\mathsf{f})(y)=}{}\times x^{-u+(-M+1)r}y^{-v+1+(-M+1)s}.
\end{gather*}
Here $r,s,u,v,M\in\mathbb{Z}$, $D=rv-su=1$, $G=\operatorname{gcd}(r,s)=1$, $a_0,a_1\in\mathbb{C}$, $a_0\ne 0$. This is embeddable into $D_1G_1E_4F_2$ by setting there $a_2=0$.
\begin{gather*}
 \mathbf{L}=\mathbf{2}\\
\pi(\mathsf{k})(x)= q^{-2s}x ,\qquad\pi(\mathsf{k})(y)=q^{2r}y ,\\
\pi(\mathsf{e})(x)= a_0\big(1-q^{v+(M-2)s}\big)x^{u+1+Mr}y^{v+Ms}+ a_2\big(1-q^{v+Ms}\big)x^{u+1+(M+2)r}y^{v+(M+2)s} ,\\
\pi(\mathsf{e})(y)= a_0\big(q^{u+Mr}-q^{2r}\big)x^{u+Mr}y^{v+1+Ms}+ a_2\big(q^{u+(M+2)r}-q^{2r}\big)x^{u+(M+2)r}y^{v+1+(M+2)s} ,\\
\pi(\mathsf{f})(x)= a_0^{-1}q^{(u+Mr)(v+Ms)+3}\big(1-q^2\big)^{-2}\big(q^{2s}-q^{-v-Ms}\big)x^{-u+1-Mr}y^{-v-Ms} ,\\
\pi(\mathsf{f})(y)= a_0^{-1}q^{(u+Mr)(v+Ms)+3}\big(1-q^2\big)^{-2}\big(q^{-u+(-M-2)r}-1\big)x^{-u-Mr}y^{-v+1-Ms} .
\end{gather*}
Here $r,s,u,v,M\in\mathbb{Z}$, $D=rv-su=1$, $G=\operatorname{gcd}(r,s)=1$, $a_0,a_2\in\mathbb{C}$, $a_0\ne 0$. This is embeddable into $D_1G_1E_3F_1$ by setting there $a_4=0$.

\subsection[The case $D=2$]{The case $\boldsymbol{D=2}$}

With $D=2$, \eqref{abgf} is only a property for a pair $\binom{\alpha}{\beta}$ of weight constants as a solution of \eqref{wce}, corresponding to a non-generic symmetry $\pi$ under consideration. Namely, one has $\alpha^2=q^{-2s}$, $\beta^2=q^{2r}$. The minimality assumption on $(r,s)$ is essential here. For example, $\binom{\alpha}{\beta}=\binom{q^{-2}}{-q^2}$ is a solution of \eqref{wce} with $\Phi=\left(\begin{smallmatrix}r & s\\ u & v\end{smallmatrix}\right)=\left(\begin{smallmatrix}2 & 2\\ 1 & 2\end{smallmatrix}\right)$. The pair $(r,s)$ is minimal with respect to $\binom{\alpha}{\beta}$, as $\alpha\beta=-1\ne 1$, and the associated symmetries are among those with $D=2$.

However, with another solution $\binom{\alpha}{\beta}=\binom{q^{-2}}{q^2}$ of \eqref{wce} corresponding to the same $\Phi$, the pair $(r,s)$ fails to be minimal. Dividing both $r$ and $s$ (respectively, the first line of $\Phi$) by $2$, we come back to the above picture with $D=1$.

Clearly, with $D=2$, the invariant $G$ can take only two values $1$ or $2$.

The following lemma clarifies the existence of pairs of weight constants providing the necessary minimality condition, hence the existence of series of symmetries written below derived by a~routine computing the coefficients.

\begin{Lemma}\label{lem_det2}
Let $\Phi=\left(\begin{smallmatrix}r & s\\ u & v\end{smallmatrix}\right)$ be an arbitrary integral matrix with $D=\det\Phi=2$.
\begin{enumerate}\itemsep=0pt
\item[$1.$] With $G=\operatorname{gcd}(r,s)=1$, there exist two pairs of weight constants $\binom{\alpha}{\beta}\!\in\!\left\{\binom{q^{-s}}{q^r};
\binom{(-1)^{-s}q^{-s}}{(-1)^rq^r}\right\}$, which are solutions of \eqref{wce}, along with the associated collections of symmetries. Such symmetries corresponding to different pairs of weight constants are non-isomorphic.
\item[$2.$]With $G=\operatorname{gcd}(r,s)=2$, there exists a single pair of weight constants $\binom{\alpha}{\beta}=\binom{(-1)^vq^{-s}}{(-1)^{-u}q^r}$, which is a solution of~\eqref{wce} such that the pair $(r,s)$ is minimal with respect to $\binom{\alpha}{\beta}$, along with the associated collection of symmetries.
\end{enumerate}
\end{Lemma}

\begin{proof} Let us treat $\Phi$ as an endomorphism of the multiplicative group $(\mathbb{C}{\setminus}\{0\})^2$. Consider the pair $\binom{\alpha}{\beta}=\binom{q^{-s}}{q^r}$. One can readily verify that this pair is a solution of \eqref{wce}. To get a complete list of solutions, we need a description of $\operatorname{Ker}\Phi$.

We start with the trivial purely computational observation that $\operatorname{Ker}\Phi\subset \Gamma\stackrel{\operatorname{def}}{=}\{-1;1\}^2$.

Use Cramer's rule to derive explicitly the integral matrix $\Phi'=\left(\begin{smallmatrix}v & -s\\ -u & r\end{smallmatrix}\right)$ such that $\Phi'\Phi=\Phi\Phi'=\left(\begin{smallmatrix}2 & 0\\ 0 & 2\end{smallmatrix}\right)$. The above inclusion property is certainly valid for $\operatorname{Ker}\Phi'$ as well, and $\Gamma$ is invariant with respect to actions of both matrices. As an immediate consequence of the definition of $\Phi'$ we deduce that
\begin{gather}\label{isk}
\Phi'\Gamma\subset\operatorname{Ker}\Phi.
\end{gather}
Note that $\Gamma$ is a 4-element group, and every its non-trivial proper subgroup is 2-element.

Since $\det\Phi=2$, at least one of the matrix elements $r$, $s$, $u$, $v$ is odd (otherwise $\det\Phi$ would be divisible by $4$). Let it be $r$, then
\begin{gather*}
\Phi'\binom{1}{-1}=\begin{pmatrix}v & -s\\ -u & r\end{pmatrix}\binom{1}{-1}=\binom{(-1)^{-s}}{(-1)^r}\ne\binom{1}{1},
\end{gather*}
and a similar argument works for the rest of the matrix elements. It follows that $\#\Phi'\Gamma\ge 2$.

This argument is also applicable to $\Phi$, that yields $\#\Phi\Gamma\ge 2$, and hence $\#\operatorname{Ker}\Phi=\frac{4}{\#\Phi\Gamma}\le 2$. The latter inequalities, together with \eqref{isk}, allow one to obtain
\begin{gather*}2\le\#\Phi'\Gamma\le\#\operatorname{Ker}\Phi\le 2,\end{gather*}
whence
\begin{gather*}\Phi'\Gamma=\operatorname{Ker}\Phi,\end{gather*}
and this subgroup is 2-element.

1. Let $G=\operatorname{gcd}(r,s)=1$. In this case at least one of the integers $r$, $s$ is odd, hence, in view of the above observations, $\binom{(-1)^{-s}}{(-1)^r}\ne\binom{1}{1}$ generates $\Phi'\Gamma=\operatorname{Ker}\Phi$. Therefore in this case \eqref{wce} has two solutions $\binom{q^{-s}}{q^r}$, $\binom{(-1)^{-s}q^{-s}}{(-1)^rq^r}$.

To see that the associated symmetries are non-isomorphic, it suffices to prove that the above two pairs of weight constants are not on the same ${\rm SL}(2,\mathbb{Z})$-orbit. Assuming the contrary, we deduce the existence of a matrix $\left(\begin{smallmatrix}k & m\\ l & n\end{smallmatrix}\right)\in {\rm SL}(2,\mathbb{Z})$ with $\left(\begin{smallmatrix}k & m\\ l & n\end{smallmatrix}\right)\binom{q^{-s}}{q^r}=\binom{(-1)^{-s}q^{-s}}{(-1)^rq^r}$, which is equivalent to
\begin{gather*}
q^{-ks+mr+s} =(-1)^{-s},\qquad q^{-ls+nr-r} =(-1)^r.
\end{gather*}
It follows that $q^{2(-ks+mr+s)}=q^{2(-ls+nr-r)}=1$, and since $q$ is not a root of $1$, one also has $-ks+mr+s=-ls+nr-r=0$. Thus we conclude that $(-1)^{-s}=(-1)^r=1$, that is both $r$ and $s$ are even, which contradicts to our assumption $G=1$.

2. Let $G=\operatorname{gcd}(r,s)=2$, that is $r=2r'$, $s=2s'$. It follows that $u$, $v$ are coprime, in particular, at least one of them is odd. We thus deduce two solutions of~\eqref{wce}:
\begin{gather*}
\binom{\alpha_1}{\beta_1}=\binom{q^{-s}}{q^r},\qquad \binom{\alpha_2}{\beta_2}=\binom{(-1)^vq^{-s}}{(-1)^{-u}q^r}.
\end{gather*}
One has
\begin{gather*}
\alpha_1^{r'}\beta_1^{s'} =q^{-sr'+rs'}=q^{2(-s'r'+r's')}=1, \\
\alpha_2^{r'}\beta_2^{s'} =(-1)^{r'v-s'u}q^{-sr'+rs'}=-1,
\end{gather*}
hence only the solution $\binom{\alpha_2}{\beta_2}=\binom{(-1)^vq^{-s}}{(-1)^{-u}q^r}$ makes the pair $(r,s)$ minimal.
\end{proof}

Here is the final list of non-generic symmetries with $D=2$, coming from adjusting the coefficients in \eqref{ex6}--\eqref{fy6} via applying \eqref{effe} together with \eqref{(ef-fe)x2}--\eqref{(ef-fe)y2}.

\subsubsection*{$\boldsymbol{D_2G_1E_1F_3(a)}$}
\begin{gather*}
\pi(\mathsf{k})(x)= q^{-s}x ,\qquad\pi(\mathsf{k})(y)=q^ry ,\\
\pi(\mathsf{e})(x)= -c_0^{-1}q^{(u+Mr)(v+Ms)+3}\big(1-q^2\big)^{-2}\big(1-q^{v+(M-1)s}\big)x^{u+1+Mr}y^{v+Ms} ,\\
 \pi(\mathsf{e})(y)= -c_0^{-1}q^{(u+Mr)(v+Ms)+3}\big(1-q^2\big)^{-2}\big(q^{u+Mr}-q^r\big)x^{u+Mr}y^{v+1+Ms} ,\\
\pi(\mathsf{f})(x)=c_0\big(q^s-q^{-v-Ms}\big)x^{-u+1-Mr}y^{-v-Ms}\\
\hphantom{\pi(\mathsf{f})(x)=}{}+ c_1\big(q^s-q^{-v+(-M+1)s}\big)x^{-u+1+(-M+1)r}y^{-v+(-M+1)s}\\
\hphantom{\pi(\mathsf{f})(x)=}{}+ c_2\big(q^s-q^{-v+(-M+2)s}\big)x^{-u+1+(-M+2)r}y^{-v+(-M+2)s} ,\\
\pi(\mathsf{f})(y)= c_0\big(q^{-u+(-M-1)r}-1\big)x^{-u-Mr}y^{-v+1-Ms}\\
\hphantom{\pi(\mathsf{f})(y)=}{} + c_1\big(q^{-u-Mr}-1\big)x^{-u+(-M+1)r}y^{-v+1+(-M+1)s}\\
\hphantom{\pi(\mathsf{f})(y)=}{}+ c_2\big(q^{-u+(-M+1)r}-1\big)x^{-u+(-M+2)r}y^{-v+1+(-M+2)s} .
\end{gather*}
Here $r,s,u,v,M\in\mathbb{Z}$, $D=rv-su=2$, $G=\operatorname{gcd}(r,s)=1$; $c_0,c_1,c_2\in\mathbb{C}$, $c_0\ne 0$. This corresponds to $L=0$.

\subsubsection*{$\boldsymbol{D_2G_1E_2F_2(a)}$}
\begin{gather*}
\pi(\mathsf{k})(x)= q^{-s}x ,\qquad\pi(\mathsf{k})(y)=q^ry ,\\
\pi(\mathsf{e})(x)= a_0\big(1-q^{v+(M-1)s}\big)x^{u+1+Mr}y^{v+Ms}+ a_1\big(1-q^{v+Ms}\big)x^{u+1+(M+1)r}y^{v+(M+1)s} ,\\
\pi(\mathsf{e})(y)= a_0\big(q^{u+Mr}-q^r\big)x^{u+Mr}y^{v+1+Ms}+ a_1\big(q^{u+(M+1)r}-q^r\big)x^{u+(M+1)r}y^{v+1+(M+1)s} ,\\
\pi(\mathsf{f})(x)= -a_0^{-1}q^{(u+Mr)(v+Ms)+3}\big(1-q^2\big)^{-2}\big(q^s-q^{-v-Ms}\big)x^{-u+1-Mr}y^{-v-Ms}\\
\hphantom{\pi(\mathsf{f})(x)=}{}+ c_1\big(q^s-q^{-v+(-M+1)s}\big)x^{-u+1+(-M+1)r}y^{-v+(-M+1)s} ,\\
\pi(\mathsf{f})(y)= -a_0^{-1}q^{(u+Mr)(v+Ms)+3}\big(1-q^2\big)^{-2}\big(q^{-u+(-M-1)r}-1\big)x^{-u-Mr}y^{-v+1-Ms}\\
\hphantom{\pi(\mathsf{f})(y)=}{}+ c_1\big(q^{-u-Mr}-1\big)x^{-u+(-M+1)r}y^{-v+1+(-M+1)s} .
\end{gather*}
Here $r,s,u,v,M\in\mathbb{Z}$, $D=rv-su=2$, $G=\operatorname{gcd}(r,s)=1$; $a_0,a_1,c_1\in\mathbb{C}$, $a_0\ne 0$. This corresponds to $L=1$.

\subsubsection*{$\boldsymbol{D_2G_1E_3F_1(a)}$}
\begin{gather*}
\pi(\mathsf{k})(x)= q^{-s}x ,\qquad\pi(\mathsf{k})(y)=q^ry ,\\
\pi(\mathsf{e})(x)= a_0\big(1-q^{v+(M-1)s}\big)x^{u+1+Mr}y^{v+Ms}+ a_1\big(1-q^{v+Ms}\big)x^{u+1+(M+1)r}y^{v+(M+1)s}\\
\hphantom{\pi(\mathsf{e})(x)=}{}+ a_2\big(1-q^{v+(M+1)s}\big)x^{u+1+(M+2)r}y^{v+(M+2)s} ,\\
\pi(\mathsf{e})(y)= a_0\big(q^{u+Mr}-q^r\big)x^{u+Mr}y^{v+1+Ms}+ a_1\big(q^{u+(M+1)r}-q^r\big)x^{u+(M+1)r}y^{v+1+(M+1)s}\\
\hphantom{\pi(\mathsf{e})(y)=}{} + a_2\big(q^{u+(M+2)r}-q^r\big)x^{u+(M+2)r}y^{v+1+(M+2)s} ,\\
\pi(\mathsf{f})(x)= a_0^{-1}q^{(u+Mr)(v+Ms)+3}\big(1-q^2\big)^{-2}\big(q^s-q^{-v-Ms}\big)x^{-u+1-Mr}y^{-v-Ms} ,\\
\pi(\mathsf{f})(y)= a_0^{-1}q^{(u+Mr)(v+Ms)+3}\big(1-q^2\big)^{-2}\big(q^{-u+(-M-1)r}-1\big)x^{-u-Mr}y^{-v+1-Ms} .
\end{gather*}
Here $r,s,u,v,M\in\mathbb{Z}$, $D=rv-su=2$, $G=\operatorname{gcd}(r,s)=1$; $a_0,a_1,a_2\in\mathbb{C}$, $a_0\ne 0$. This corresponds to $L=2$.

\subsubsection*{$\boldsymbol{D_2G_1E_1F_3(b)}$}
\begin{gather*}
\pi(\mathsf{k})(x)= (-1)^{-s}q^{-s}x ,\qquad\pi(\mathsf{k})(y)=(-1)^rq^ry ,\\
\pi(\mathsf{e})(x)= -c_0^{-1}q^{(u+Mr)(v+Ms)+3}\big(1-q^2\big)^{-2}\big(1-(-1)^{-s}q^{v+(M-1)s}\big)x^{u+1+Mr}y^{v+Ms} ,\\
\pi(\mathsf{e})(y)= -c_0^{-1}q^{(u+Mr)(v+Ms)+3}\big(1-q^2\big)^{-2}\big(q^{u+Mr}-(-1)^rq^r\big)x^{u+Mr}y^{v+1+Ms} ,\\
\pi(\mathsf{f})(x)= c_0\big((-1)^sq^s-q^{-v-Ms}\big)x^{-u+1-Mr}y^{-v-Ms}\\
\hphantom{\pi(\mathsf{f})(x)=}{}+ c_1\big((-1)^sq^s-q^{-v+(-M+1)s}\big)x^{-u+1+(-M+1)r}y^{-v+(-M+1)s}\\
\hphantom{\pi(\mathsf{f})(x)=}{}+ c_2\big((-1)^sq^s-q^{-v+(-M+2)s}\big)x^{-u+1+(-M+2)r}y^{-v+(-M+2)s} ,\\
\pi(\mathsf{f})(y)= c_0\big((-1)^{-r}q^{-u+(-M-1)r}-1\big)x^{-u-Mr}y^{-v+1-Ms}\\
\hphantom{\pi(\mathsf{f})(y)=}{}+ c_1\big((-1)^{-r}q^{-u-Mr}-1\big)x^{-u+(-M+1)r}y^{-v+1+(-M+1)s}\\
\hphantom{\pi(\mathsf{f})(y)=}{}+ c_2\big((-1)^{-r}q^{-u+(-M+1)r}-1\big)x^{-u+(-M+2)r}y^{-v+1+(-M+2)s} .
\end{gather*}
Here $r,s,u,v,M\in\mathbb{Z}$, $D=rv-su=2$, $G=\operatorname{gcd}(r,s)=1$; $c_0,c_1,c_2\in\mathbb{C}$, $c_0\ne 0$. This corresponds to $L=0$.

\subsubsection*{$\boldsymbol{D_2G_1E_2F_2(b)}$}
\begin{gather*}
\pi(\mathsf{k})(x)= (-1)^{-s}q^{-s}x ,\qquad\pi(\mathsf{k})(y)=(-1)^rq^ry ,\\
\pi(\mathsf{e})(x)= a_0\big(1-(-1)^{-s}q^{v+(M-1)s}\big)x^{u+1+Mr}y^{v+Ms}\\
\hphantom{\pi(\mathsf{e})(x)=}{} + a_1\big(1-(-1)^{-s}q^{v+Ms}\big)x^{u+1+(M+1)r}y^{v+(M+1)s} ,\\
\pi(\mathsf{e})(y)= a_0\big(q^{u+Mr}-(-1)^rq^r\big)x^{u+Mr}y^{v+1+Ms}\\
\hphantom{\pi(\mathsf{e})(y)=}{} + a_1\big(q^{u+(M+1)r}-(-1)^rq^r\big)x^{u+(M+1)r}y^{v+1+(M+1)s} ,\\
\pi(\mathsf{f})(x)= -a_0^{-1}q^{(u+Mr)(v+Ms)+3}\big(1-q^2\big)^{-2}\big((-1)^sq^s-q^{-v-Ms}\big)x^{-u+1-Mr}y^{-v-Ms}\\
\hphantom{\pi(\mathsf{f})(x)=}{} + c_1\big((-1)^sq^s-q^{-v+(-M+1)s}\big)x^{-u+1+(-M+1)r}y^{-v+(-M+1)s} ,\\
\pi(\mathsf{f})(y)= -a_0^{-1}q^{(u+Mr)(v+Ms)+3}\big(1-q^2\big)^{-2}\big((-1)^{-r}q^{-u+(-M-1)r}-1\big) x^{-u-Mr}y^{-v+1-Ms}\\
\hphantom{\pi(\mathsf{f})(y)=}{}+ c_1\big((-1)^{-r}q^{-u-Mr}-1\big)x^{-u+(-M+1)r}y^{-v+1+(-M+1)s} .
\end{gather*}
Here $r,s,u,v,M\in\mathbb{Z}$, $D=rv-su=2$, $G=\operatorname{gcd}(r,s)=1$; $a_0,a_1,c_1\in\mathbb{C}$, $a_0\ne 0$. This corresponds to $L=1$.

\subsubsection*{$\boldsymbol{D_2G_1E_3F_1(b)}$}
\begin{gather*}
\pi(\mathsf{k})(x)= (-1)^{-s}q^{-s}x ,\qquad\pi(\mathsf{k})(y)=(-1)^rq^ry ,\\
\pi(\mathsf{e})(x)= a_0\big(1-(-1)^{-s}q^{v+(M-1)s}\big)x^{u+1+Mr}y^{v+Ms}\\
\hphantom{\pi(\mathsf{e})(x)=}{} + a_1\big(1-(-1)^{-s}q^{v+Ms}\big)x^{u+1+(M+1)r}y^{v+(M+1)s}\\
\hphantom{\pi(\mathsf{e})(x)=}{} + a_2\big(1-(-1)^{-s}q^{v+(M+1)s}\big)x^{u+1+(M+2)r}y^{v+(M+2)s} ,\\
\pi(\mathsf{e})(y)= a_0\big(q^{u+Mr}-(-1)^rq^r\big)x^{u+Mr}y^{v+1+Ms}\\
\hphantom{\pi(\mathsf{e})(y)=}{} + a_1\big(q^{u+(M+1)r}-(-1)^rq^r\big)x^{u+(M+1)r}y^{v+1+(M+1)s}\\
\hphantom{\pi(\mathsf{e})(y)=}{} + a_2\big(q^{u+(M+2)r}-(-1)^rq^r\big)x^{u+(M+2)r}y^{v+1+(M+2)s} ,\\
\pi(\mathsf{f})(x)= a_0^{-1}q^{(u+Mr)(v+Ms)+3}\big(1-q^2\big)^{-2}\big((-1)^sq^s-q^{-v-Ms}\big)x^{-u+1-Mr}y^{-v-Ms} ,\\
\pi(\mathsf{f})(y)= a_0^{-1}q^{(u+Mr)(v+Ms)+3}\big(1-q^2\big)^{-2}\big((-1)^{-r}q^{-u+(-M-1)r}-1\big)x^{-u-Mr}y^{-v+1-Ms} .
\end{gather*}
Here $r,s,u,v,M\in\mathbb{Z}$, $D=rv-su=2$, $G=\operatorname{gcd}(r,s)=1$; $a_0,a_1,a_2\in\mathbb{C}$, $a_0\ne 0$. This corresponds to $L=2$.

\subsubsection*{$\boldsymbol{D_2G_2E_1F_3}$}
\begin{gather*}
\pi(\mathsf{k})(x)= (-1)^vq^{-s}x ,\qquad\pi(\mathsf{k})(y)=(-1)^{-u}q^ry ,\\
\pi(\mathsf{e})(x)= -c_0^{-1}q^{(u+Mr)(v+Ms)+3}\big(1-q^2\big)^{-2}\big(1-(-1)^vq^{v+(M-1)s}\big)x^{u+1+Mr}y^{v+Ms} ,\\
\pi(\mathsf{e})(y)= -c_0^{-1}q^{(u+Mr)(v+Ms)+3}\big(1-q^2\big)^{-2}\big(q^{u+Mr}-(-1)^{-u}q^r\big)x^{u+Mr}y^{v+1+Ms} ,\\
\pi(\mathsf{f})(x)= c_0\big((-1)^{-v}q^s-q^{-v-Ms}\big)x^{-u+1-Mr}y^{-v-Ms}\\
\hphantom{\pi(\mathsf{f})(x)=}{}+ c_1\big((-1)^{-v}q^s-q^{-v+(-M+1)s}\big)x^{-u+1+(-M+1)r}y^{-v+(-M+1)s}\\
\hphantom{\pi(\mathsf{f})(x)=}{}+ c_2\big((-1)^{-v}q^s-q^{-v+(-M+2)s}\big)x^{-u+1+(-M+2)r}y^{-v+(-M+2)s} ,\\
\pi(\mathsf{f})(y)= c_0\big((-1)^uq^{-u+(-M-1)r}-1\big)x^{-u-Mr}y^{-v+1-Ms}\\
\hphantom{\pi(\mathsf{f})(y)=}{} + c_1\big((-1)^uq^{-u-Mr}-1\big)x^{-u+(-M+1)r}y^{-v+1+(-M+1)s}\\
\hphantom{\pi(\mathsf{f})(y)=}{} + c_2\big((-1)^uq^{-u+(-M+1)r}-1\big)x^{-u+(-M+2)r}y^{-v+1+(-M+2)s} .
\end{gather*}
Here $r,s,u,v,M\in\mathbb{Z}$, $D=rv-su=2$, $G=\operatorname{gcd}(r,s)=2$; $c_0,c_1,c_2\in\mathbb{C}$, $c_0\ne 0$. This corresponds to $L=0$.

\subsubsection*{$\boldsymbol{D_2G_2E_2F_2}$}
\begin{gather*}
\pi(\mathsf{k})(x)= (-1)^vq^{-s}x ,\qquad\pi(\mathsf{k})(y)=(-1)^{-u}q^ry ,\\
\pi(\mathsf{e})(x)= a_0\big(1-(-1)^vq^{v+(M-1)s}\big)x^{u+1+Mr}y^{v+Ms}\\
\hphantom{\pi(\mathsf{e})(x)=}{}+ a_1\big(1-(-1)^vq^{v+Ms}\big)x^{u+1+(M+1)r}y^{v+(M+1)s} ,\\
\pi(\mathsf{e})(y)= a_0\big(q^{u+Mr}-(-1)^{-u}q^r\big)x^{u+Mr}y^{v+1+Ms}\\
\hphantom{\pi(\mathsf{e})(y)=}{} + a_1\big(q^{u+(M+1)r}-(-1)^{-u}q^r\big)x^{u+(M+1)r}y^{v+1+(M+1)s} ,\\
\pi(\mathsf{f})(x)= -a_0^{-1}q^{(u+Mr)(v+Ms)+3}\big(1-q^2\big)^{-2}\big((-1)^{-v}q^s-q^{-v-Ms}\big)x^{-u+1-Mr}y^{-v-Ms}\\
\hphantom{\pi(\mathsf{f})(x)=}{} + c_1\big((-1)^{-v}q^s-q^{-v+(-M+1)s}\big)x^{-u+1+(-M+1)r}y^{-v+(-M+1)s} ,\\
\pi(\mathsf{f})(y)= -a_0^{-1}q^{(u+Mr)(v+Ms)+3}\big(1-q^2\big)^{-2}\big((-1)^uq^{-u+(-M-1)r}-1\big) x^{-u-Mr}y^{-v+1-Ms}\\
\hphantom{\pi(\mathsf{f})(y)=}{} + c_1\big((-1)^uq^{-u-Mr}-1\big)x^{-u+(-M+1)r}y^{-v+1+(-M+1)s} .
\end{gather*}
Here $r,s,u,v,M\in\mathbb{Z}$, $D=rv-su=2$, $G=\operatorname{gcd}(r,s)=2$; $a_0,a_1,c_1\in\mathbb{C}$, $a_0\ne 0$. This corresponds to $L=1$.

\subsubsection*{$\boldsymbol{D_2G_2E_3F_1}$}
\begin{gather*}
\pi(\mathsf{k})(x)= (-1)^vq^{-s}x ,\qquad\pi(\mathsf{k})(y)=(-1)^{-u}q^ry ,\\
\pi(\mathsf{e})(x)= a_0\big(1-(-1)^vq^{v+(M-1)s}\big)x^{u+1+Mr}y^{v+Ms}\\
\hphantom{\pi(\mathsf{e})(x)=}{}+ a_1\big(1-(-1)^vq^{v+Ms}\big)x^{u+1+(M+1)r}y^{v+(M+1)s}\\
\hphantom{\pi(\mathsf{e})(x)=}{}+ a_2\big(1-(-1)^vq^{v+(M+1)s}\big)x^{u+1+(M+2)r}y^{v+(M+2)s} ,\\
\pi(\mathsf{e})(y)= a_0\big(q^{u+Mr}-(-1)^{-u}q^r\big)x^{u+Mr}y^{v+1+Ms}\\
\hphantom{\pi(\mathsf{e})(y)=}{} + a_1\big(q^{u+(M+1)r}-(-1)^{-u}q^r\big)x^{u+(M+1)r}y^{v+1+(M+1)s}\\
\hphantom{\pi(\mathsf{e})(y)=}{} + a_2\big(q^{u+(M+2)r}-(-1)^{-u}q^r\big)x^{u+(M+2)r}y^{v+1+(M+2)s} ,\\
\pi(\mathsf{f})(x)= a_0^{-1}q^{(u+Mr)(v+Ms)+3}\big(1-q^2\big)^{-2}\big((-1)^{-v}q^s-q^{-v-Ms}\big)x^{-u+1-Mr}y^{-v-Ms} ,\\
\pi(\mathsf{f})(y)= a_0^{-1}q^{(u+Mr)(v+Ms)+3}\big(1-q^2\big)^{-2}\big((-1)^uq^{-u+(-M-1)r}-1\big)x^{-u-Mr}y^{-v+1-Ms} .
\end{gather*}
Here $r,s,u,v,M\in\mathbb{Z}$, $D=rv-su=2$, $G=\operatorname{gcd}(r,s)=2$; $a_0,a_1,a_2\in\mathbb{C}$, $a_0\ne 0$. This corresponds to $L=2$.

\subsubsection*{The case $\boldsymbol{D=2}$, $\boldsymbol{N=1}$}

Under the assumption $D=2$, by Lemma~\ref{iD} we have also to consider the case $N=1$. To see that this way no additional symmetries occur, one has to compute coefficients just to write down the corresponding series. We stick here to the case $G=1$, $\alpha=q^{-s}$, $\beta=q^r$; the rest of cases are to be considered in a similar way. No names are given to these series, because they are all embeddable to the above series with $N=2$.
\begin{gather*}
 \mathbf{L}=\mathbf{0}\\
\pi(\mathsf{k})(x)= q^{-s}x ,\qquad\pi(\mathsf{k})(y)=q^ry ,\\
\pi(\mathsf{e})(x)= -c_0^{-1}q^{(u+Mr)(v+Ms)+3}\big(1-q^2\big)^{-2}\big(1-q^{v+(M-1)s}\big)x^{u+1+Mr}y^{v+Ms} ,\\
\pi(\mathsf{e})(y)= -c_0^{-1}q^{(u+Mr)(v+Ms)+3}\big(1-q^2\big)^{-2}\big(q^{u+Mr}-q^r\big)x^{u+Mr}y^{v+1+Ms} ,\\
\pi(\mathsf{f})(x)= c_0\big(q^s-q^{-v-Ms}\big)x^{-u+1-Mr}y^{-v-Ms}\\
\hphantom{\pi(\mathsf{f})(x)=}{} + c_1\big(q^s-q^{-v+(-M+1)s}\big) x^{-u+1+(-M+1)r}y^{-v+(-M+1)s} ,\\
 \pi(\mathsf{f})(y)= c_0\big(q^{-u+(-M-1)r}-1\big)x^{-u-Mr}y^{-v+1-Ms}\\
 \hphantom{\pi(\mathsf{f})(y)=}{} + c_1\big(q^{-u-Mr}-1\big)x^{-u+(-M+1)r}y^{-v+1+(-M+1)s} .
\end{gather*}
Here $r,s,u,v,M\in\mathbb{Z}$, $D=rv-su=2$, $G=\operatorname{gcd}(r,s)=1$; $c_0,c_1\in\mathbb{C}$, $c_0\ne 0$. This is embeddable into $D_2G_1E_1F_3(a)$ by setting there $c_2=0$.
\begin{gather*}
 \mathbf{L}=\mathbf{1}\\
 \pi(\mathsf{k})(x)= q^{-s}x ,\qquad\pi(\mathsf{k})(y)=q^ry ,\\
\pi(\mathsf{e})(x)= a_0\big(1-q^{v+(M-1)s}\big)x^{u+1+Mr}y^{v+Ms}+\\ a_1\big(1-q^{v+Ms}\big)x^{u+1+(M+1)r}y^{v+(M+1)s} ,\\
\pi(\mathsf{e})(y)= a_0\big(q^{u+Mr}-q^r\big)x^{u+Mr}y^{v+1+Ms}+ a_1\big(q^{u+(M+1)r}-q^r\big)x^{u+(M+1)r}y^{v+1+(M+1)s} ,\\
\pi(\mathsf{f})(x)= a_0^{-1}q^{(u+Mr)(v+Ms)+3}\big(1-q^2\big)^{-2}\big(q^s-q^{-v-Ms}\big)x^{-u+1-Mr}y^{-v-Ms} ,\\
\pi(\mathsf{f})(y)= a_0^{-1}q^{(u+Mr)(v+Ms)+3}\big(1-q^2\big)^{-2}\big(q^{-u+(-M-1)r}-1\big)x^{-u-Mr}y^{-v+1-Ms} .
\end{gather*}
Here $r,s,u,v,M\in\mathbb{Z}$, $D=rv-su=2$, $G=\operatorname{gcd}(r,s)=1$; $a_0,a_1\in\mathbb{C}$, $a_0\ne 0$. This is embeddable into $D_2G_1E_3F_1(a)$ by setting there $a_2=0$.

\subsection[The case $D=4$]{The case $\boldsymbol{D=4}$}

With $D=4$, Lemma~\ref{iD} implies that the only possible value for $N$ is $N=1$. Thus, the multiplier(s) $q^{-2+iD}-1$ in \eqref{(ef-fe)x2}, \eqref{(ef-fe)y2} at the extreme monomials corresponding to $i=N=1$ are non-zero. It follows that $\alpha q^s-\alpha^{-1}=\beta^{-1}q^r-\beta=0$, that is, the weight constants are subject to
\begin{gather}\label{sbt}
\alpha^2=q^{-s},\qquad\beta^2=q^r,
\end{gather}
which is a more subtle condition than $\alpha^4=q^{-2s}$, $\beta^4=q^{2r}$ coming from \eqref{abgf}. The existence of solutions of \eqref{wce} compatible to \eqref{sbt} with the minimality property for $(r,s)$ (hence the existence of symmetries) is described by

\begin{Lemma}\label{lem_det4}
Let the integral matrix (of integral parameters of symmetries) $\Phi=\left(\begin{smallmatrix}r & s\\ u & v\end{smallmatrix}\right)$ be such that $D=\det\Phi=4$. Then
\begin{enumerate}\itemsep=0pt
\item[$1.$] Let $G=\operatorname{gcd}(r,s)=1$ and $\zeta$ be a fixed square root of $q$ $(\zeta^2=q)$, then there exist exactly two solutions $\binom{\alpha}{\beta}$ of~\eqref{wce} compatible with~\eqref{sbt}
 \begin{gather*}
 \binom{\zeta^{-s}}{\zeta^r},\qquad\binom{(-1)^s\zeta^{-s}}{(-1)^r\zeta^r}
 \end{gather*}
$($the latter pair is nothing more than the solution corresponding to $-\zeta$, another square root of~$q)$. Hence the existence of symmetries listed below by calculation of coefficients. Any two symmetries corresponding to different pairs of weight constants as above are non-isomorphic.
\item[$2.$] With $G=\operatorname{gcd}(r,s)=2$, the solutions $\binom{\alpha}{\beta}$ of \eqref{wce} compatible with \eqref{sbt} with the minima\-li\-ty property for $(r,s)$ exist if and only if both $u$ and $v$ are even. In this case there exist two such distinct solutions as follows{\samepage
 \begin{gather*}
 \binom{-q^{-s'}}{(-1)^{r'+1}q^{r'}},\qquad \binom{(-1)^{s'+1}q^{-s'}}{-q^{r'}},
 \end{gather*}
 where $r=2r'$, $s=2s'$. Hence the existence of symmetries listed below.}
\item[$3.$] In the case $G=\operatorname{gcd}(r,s)=4$, there exist no solutions $\binom{\alpha}{\beta}$ of \eqref{wce} compatible with \eqref{sbt} with the minimality property for $(r,s)$. Hence no associated symmetries.
\end{enumerate}
\end{Lemma}

\begin{proof} 1. Let $G=\operatorname{gcd}(r,s)=1$. Since $r$ and $s$ are coprime, there exist $u',v'\in\mathbb{Z}$ such that $v'r-u's=1$. This, together with $rv-su=4$, implies
\begin{gather*}r(v-4v')=s(u-4u').\end{gather*}
One more application of the fact that $r$ and $s$ are coprime allows one to deduce the existence of $t\in\mathbb{Z}$ with $u-4u'=tr$, $v-4v'=ts$. It has been mentioned in Section~\ref{wpem} that (obviously) a replacement of $u$ and $v$ by $u-tr$ and $v-ts$, respectively, does not affect the set of solutions (hence the set of solutions with the minimality property for $(r,s)$) $\binom{\alpha}{\beta}$ of~\eqref{wce}. So we may assume that $u=4u'$, $v=4v'$.

Let us fix a square root $\zeta$ of $q$ ($\zeta^2=q$), and consider the pair $\binom{\alpha}{\beta}=\binom{\zeta^{-s}}{\zeta^r}$. An easy verification shows that this pair is a solution of \eqref{wce} compatible with \eqref{sbt}. Observe that any pair of constants subject to \eqref{sbt} differs from $\binom{\zeta^{-s}}{\zeta^r}$ by changing signs of the components, and there exist exactly $4$ such pairs. Let us present a list of those as follows
\begin{gather*}
\binom{\zeta^{-s}}{\zeta^r},\qquad\binom{(-1)^s\zeta^{-s}}{(-1)^r\zeta^r}, \qquad\binom{-\zeta^{-s}}{(-1)^{r+1}\zeta^r},\qquad
\binom{(-1)^{s+1}\zeta^{-s}}{-\zeta^r}.
\end{gather*}
A routine verification demonstrates that for any coprime $r$, $s$ the elements of this list are pairwise distinct. Another simple calculation shows that the initial two pairs are solutions of~\eqref{wce}, while the latter two pairs are subject to $\alpha^r\beta^s=-1$. The last claim is due to our assumption that $r$, $s$ are coprime, so that $r+rs+s$ cannot be even. That is, only the initial two pairs are solutions of~\eqref{wce}.

To prove that any two symmetries corresponding to different pairs of weight constants $\binom{\zeta^{-s}}{\zeta^r}$, $\binom{(-1)^s\zeta^{-s}}{(-1)^r\zeta^r}$ are non-isomorphic, it suffices to demonstrate that these pairs of constants are not on the same ${\rm SL}(2,\mathbb{Z})$-orbit. To see this, let us observe first that $\zeta$ is obviously not a~root of~$1$, together with~$q$. If one replaces here~$\zeta$ with~$q$, we get the two pairs considered in the proof of Lemma~\ref{lem_det2}. These pairs are proved there to be not on the same ${\rm SL}(2,\mathbb{Z})$-orbit, and the argument used works also in our present case.

2. Let $G=\operatorname{gcd}(r,s)=2$. Then $r=2r'$, $s=2s'$, and (in our context) $r'$ and $s'$ are coprime. Any pair of weight constants $\alpha$, $\beta$ subject to \eqref{sbt} has the form $\alpha=\varepsilon_\alpha q^{-s'}$, $\beta=\varepsilon_\beta q^{r'}$, with $\varepsilon_\alpha,\varepsilon_\beta=\pm 1$.

Suppose that both $u$ and $v$ are even. Then for any choice of $\varepsilon_\alpha$, $\varepsilon_\beta$ one has
\begin{gather*}
\alpha^r\beta^s=\varepsilon_\alpha^r\varepsilon_\beta^sq^{-s'r+r's}=\varepsilon_\alpha^{2r'}\varepsilon_\beta^{2s'}q^{-2s'r'+2r's'}=1,\\
\alpha^u\beta^v=\varepsilon_\alpha^u\varepsilon_\beta^vq^{-s'u+r'v}=q^2.\end{gather*}
This means that $\binom{\alpha}{\beta}$ is a solution of \eqref{wce} for any $\varepsilon_\alpha$, $\varepsilon_\beta$, and one has $4$ such solutions. We need only to distinguish those making $(r,s)$ minimal. For that, we reproduce the idea used in the previous case with $G=1$ in writing down the $4$ pairs of constants as follows
\begin{gather*}
\binom{q^{-s'}}{q^{r'}},\qquad\binom{(-1)^{s'}q^{-s'}}{(-1)^{r'}q^{r'}},\qquad \binom{-q^{-s'}}{(-1)^{r'+1}q^{r'}},\qquad\binom{(-1)^{s'+1}q^{-s'}}{-q^{r'}}.
\end{gather*}
It turns out that, given any coprime $r'$, $s'$ (as in our case), the elements of this list are pairwise distinct, which is a matter of routine verification. Also, the above discussion demonstrates that these pairs of constants are solutions of~\eqref{wce} subject to~\eqref{sbt}.

As for minimality for $(r,s)$, an easy calculation with $\binom{\alpha}{\beta}$ standing for initial two pairs of constants establishes that $\alpha^{r'}\beta^{s'}=1$, thus minimality condition fails. However, a similar simple calculation based on $r'$, $s'$ being coprime in the case of the last two pairs of constants $\binom{\alpha}{\beta}$ shows that $\alpha^{r'}\beta^{s'}=-1$, hence our claim.

Conversely, suppose that there exists a solution $\binom{\alpha}{\beta}$ of \eqref{wce} compatible with \eqref{sbt}, with the pair $(r,s)$ being minimal. In view of the above observations, $\alpha=\varepsilon_\alpha q^{-s'}$, $\beta=\varepsilon_\beta q^{r'}$ for some $\varepsilon_\alpha,\varepsilon_\beta\in\{-1;1\}$, $\alpha^u\beta^v=\varepsilon_\alpha^u\varepsilon_\beta^v q^2$, $\alpha^{r'}\beta^{s'}=\varepsilon_\alpha^{r'}\varepsilon_\beta^{s'}$. Our current assumption on $\alpha$, $\beta$ allows one to conclude that, with $\Phi'=\left(\begin{smallmatrix}r' & s'\\ u & v\end{smallmatrix}\right)$, one has
\begin{gather*}
\binom{-1}{q^2}=\Phi'\binom{\alpha}{\beta}=
\Phi'\binom{\varepsilon_\alpha}{\varepsilon_\beta}\Phi'\binom{q^{-s'}}{q^{r'}}
=\Phi'\binom{\varepsilon_\alpha}{\varepsilon_\beta}\cdot\binom{1}{q^2},
\end{gather*}
whence
\begin{gather}\label{phi_eps}
\Phi'\binom{\varepsilon_\alpha}{\varepsilon_\beta}=\binom{-1}1. \end{gather}

As $\det\Phi'=2$, an application of Cramer's rule produces an integral matrix $\Phi''$ such that $\Phi''\Phi'=\left(\begin{smallmatrix}2 & 0\\ 0 & 2\end{smallmatrix}\right)$. With $\Phi''$ being applied to \eqref{phi_eps}, one has $\left(\begin{smallmatrix}2 & 0\\ 0 & 2\end{smallmatrix}\right) \binom{\varepsilon_\alpha}{\varepsilon_\beta}=\Phi''\binom{-1}1$. This relation, using the explicit form of $\Phi''=\left(\begin{smallmatrix}v & -s'\\ -u & r'\end{smallmatrix}\right)$, becomes $\binom11=\binom{(-1)^v}{(-1)^{-u}}$, which implies that both $u$ and $v$ are even.

3. Let $G=\operatorname{gcd}(r,s)=4$, in particular $r=4r'$, $s=4s'$. Suppose $\binom{\alpha}{\beta}$ is a solution of~\eqref{wce} compatible with~\eqref{sbt}. Then one has
\begin{gather*}\alpha^{2r'}\beta^{2s'}=q^{-sr'+rs'}=q^{4(-s'r'+r's')}=1,\end{gather*}
which implies that the pair $(r,s)$ is not minimal.
\end{proof}

\begin{Remark}It should be noted that in Lemma~\ref{lem_det4}(2), the choice between the two distinct pairs of weight constants in general does not lead to non-isomorphic symmetries, as it was the case in Lemmas~\ref{lem_det2}(1) and~\ref{lem_det4}(1). For example, substituting there $r'=0$, $s'=1$ one observes that the distinct pairs of weight constants mentioned in the statement of (2) are on the same ${\rm SL}(2,\mathbb{Z})$-orbit:
\begin{gather*}
\begin{pmatrix}1 & 1\\ 0 & 1\end{pmatrix}\binom{-q^{-1}}{-1}=\binom{q^{-1}}{-1}.
\end{gather*}
Hence an application of the isomorphism $\varphi_{\sigma,1,1}$ with $\sigma= \left(\begin{smallmatrix}1 & 1\\ 0 & 1\end{smallmatrix}\right)$ intertwines the associated symmetries.
\end{Remark}

Here is the final list of non-generic symmetries with $D=4$, coming from adjusting the coefficients in \eqref{ex6}--\eqref{fy6} via applying \eqref{effe} together with \eqref{(ef-fe)x2}--\eqref{(ef-fe)y2}.

The initial $4$ series assume a square root $\zeta$ of $q$ ($\zeta^2=q$) being fixed.

\subsubsection*{$\boldsymbol{D_4G_1E_1F_2(a)}$}
\begin{gather*}
\pi(\mathsf{k})(x)= \zeta^{-s}x ,\qquad\pi(\mathsf{k})(y)=\zeta^ry ,\\
\pi(\mathsf{e})(x)= -c_0^{-1}q^{(u+Mr)(v+Ms)+3}\big(1-q^2\big)^{-2}\big(1-\zeta^{-s}q^{v+Ms}\big)x^{u+1+Mr}y^{v+Ms} ,\\
\pi(\mathsf{e})(y)= -c_0^{-1}q^{(u+Mr)(v+Ms)+3}\big(1-q^2\big)^{-2}\big(q^{u+Mr}-\zeta^r\big)x^{u+Mr}y^{v+1+Ms} ,\\
\pi(\mathsf{f})(x)= c_0\big(\zeta^s-q^{-v-Ms}\big)x^{-u+1-Mr}y^{-v-Ms}\\
\hphantom{\pi(\mathsf{f})(x)=}{} + c_1\big(\zeta^s-q^{-v+(-M+1)s}\big)x^{-u+1+(-M+1)r}y^{-v+(-M+1)s} ,\\
\pi(\mathsf{f})(y)= c_0\big(\zeta^{-r}q^{-u-Mr}-1\big)x^{-u-Mr}y^{-v+1-Ms}\\
\hphantom{\pi(\mathsf{f})(y)=}{} + c_1\big(\zeta^{-r}q^{-u+(-M+1)r}-1\big)x^{-u+(-M+1)r}y^{-v+1+(-M+1)s} .
\end{gather*}
Here $r,s,u,v,M\in\mathbb{Z}$, $D=rv-su=4$; $G=\operatorname{gcd}(r,s)=1$; $c_0,c_1\in\mathbb{C}$, $c_0\ne 0$. This corresponds to $L=0$.

\subsubsection*{$\boldsymbol{D_4G_1E_2F_1(a)}$}
\begin{gather*}
\pi(\mathsf{k})(x)= \zeta^{-s}x ,\qquad\pi(\mathsf{k})(y)=\zeta^ry ,\\
\pi(\mathsf{e})(x)= a_0\big(1-\zeta^{-s}q^{v+Ms}\big)x^{u+1+Mr}y^{v+Ms}\\
\hphantom{\pi(\mathsf{e})(x)=}{}+ a_1\big(1-\zeta^{-s}q^{v+(M+1)s}\big)x^{u+1+(M+1)r}y^{v+(M+1)s} ,\\
\pi(\mathsf{e})(y)= a_0\big(q^{u+Mr}-\zeta^r\big)x^{u+Mr}y^{v+1+Ms}+ a_1\big(q^{u+(M+1)r}-\zeta^r\big)x^{u+(M+1)r}y^{v+1+(M+1)s} ,\\
\pi(\mathsf{f})(x)= a_0^{-1}q^{(u+Mr)(v+Ms)+3}\big(1-q^2\big)^{-2}\big(\zeta^s-q^{-v-Ms}\big)x^{-u+1-Mr}y^{-v-Ms} ,\\
\pi(\mathsf{f})(y)= a_0^{-1}q^{(u+Mr)(v+Ms)+3}\big(1-q^2\big)^{-2}\big(\zeta^{-r}q^{-u-Mr}-1\big)x^{-u-Mr}y^{-v+1-Ms} .
\end{gather*}
Here $r,s,u,v,M\in\mathbb{Z}$, $D=rv-su=4$; $G=\operatorname{gcd}(r,s)=1$; $a_0,a_1\in\mathbb{C}$, $a_0\ne 0$. This corresponds to $L=1$.

\subsubsection*{$\boldsymbol{D_4G_1E_1F_2(b)}$}
\begin{gather*}
\pi(\mathsf{k})(x)= (-1)^s\zeta^{-s}x ,\qquad\pi(\mathsf{k})(y)=(-1)^r\zeta^ry ,\\
\pi(\mathsf{e})(x)= -c_0^{-1}q^{(u+Mr)(v+Ms)+3}\big(1-q^2\big)^{-2}\big(1-(-1)^s\zeta^{-s}q^{v+Ms}\big)x^{u+1+Mr}y^{v+Ms} ,\\
\pi(\mathsf{e})(y)= -c_0^{-1}q^{(u+Mr)(v+Ms)+3}\big(1-q^2\big)^{-2}\big(q^{u+Mr}-(-1)^r\zeta^r\big)x^{u+Mr}y^{v+1+Ms} ,\\
\pi(\mathsf{f})(x)= c_0\big((-1)^s\zeta^s-q^{-v-Ms}\big)x^{-u+1-Mr}y^{-v-Ms}\\
\hphantom{\pi(\mathsf{f})(x)=}{}+ c_1\big((-1)^s\zeta^s-q^{-v+(-M+1)s}\big)x^{-u+1+(-M+1)r}y^{-v+(-M+1)s} ,\\
\pi(\mathsf{f})(y)= c_0\big((-1)^r\zeta^{-r}q^{-u-Mr}-1\big)x^{-u-Mr}y^{-v+1-Ms}\\
\hphantom{\pi(\mathsf{f})(y)=}{} + c_1\big((-1)^r\zeta^{-r}q^{-u+(-M+1)r}-1\big)x^{-u+(-M+1)r}y^{-v+1+(-M+1)s} .
\end{gather*}
Here $r,s,u,v,M\in\mathbb{Z}$, $D=rv-su=4$; $G=\operatorname{gcd}(r,s)=1$; $c_0,c_1\in\mathbb{C}$, $c_0\ne 0$. This corresponds to $L=0$.

\subsubsection*{$\boldsymbol{D_4G_1E_2F_1(b)}$}
\begin{gather*}
\pi(\mathsf{k})(x)= (-1)^s\zeta^{-s}x ,\qquad\pi(\mathsf{k})(y)=(-1)^r\zeta^ry ,\\
\pi(\mathsf{e})(x)= a_0\big(1-(-1)^s\zeta^{-s}q^{v+Ms}\big)x^{u+1+Mr}y^{v+Ms}\\
\hphantom{\pi(\mathsf{e})(x)=}{} + a_1\big(1-(-1)^s\zeta^{-s}q^{v+(M+1)s}\big)x^{u+1+(M+1)r}y^{v+(M+1)s} ,\\
\pi(\mathsf{e})(y)= a_0\big(q^{u+Mr}-(-1)^r\zeta^r\big)x^{u+Mr}y^{v+1+Ms}\\
\hphantom{\pi(\mathsf{e})(y)=}{} + a_1\big(q^{u+(M+1)r}-(-1)^r\zeta^r\big)x^{u+(M+1)r}y^{v+1+(M+1)s} ,\\
\pi(\mathsf{f})(x)= a_0^{-1}q^{(u+Mr)(v+Ms)+3}\big(1-q^2\big)^{-2}\big((-1)^s\zeta^s-q^{-v-Ms}\big)x^{-u+1-Mr}y^{-v-Ms} ,\\
\pi(\mathsf{f})(y)= a_0^{-1}q^{(u+Mr)(v+Ms)+3}\big(1-q^2\big)^{-2}\big((-1)^r\zeta^{-r}q^{-u-Mr}-1\big)x^{-u-Mr}y^{-v+1-Ms} .
\end{gather*}
Here $r,s,u,v,M\in\mathbb{Z}$, $D=rv-su=4$; $G=\operatorname{gcd}(r,s)=1$; $a_0,a_1\in\mathbb{C}$, $a_0\ne 0$. This corresponds to $L=1$.

The final $4$ series assume even integral parameters $r=2r'$, $s=2s'$.

\subsubsection*{$\boldsymbol{D_4G_2E_1F_2(a)}$}
\begin{gather*}
\pi(\mathsf{k})(x)= -q^{-s'}x ,\qquad\pi(\mathsf{k})(y)=(-1)^{r'+1}q^{r'}y ,\\
\pi(\mathsf{e})(x)= -c_0^{-1}q^{(u+Mr)(v+Ms)+3}\big(1-q^2\big)^{-2}\big(1+q^{-s'+v+Ms}\big)x^{u+1+Mr}y^{v+Ms} ,\\
\pi(\mathsf{e})(y)= -c_0^{-1}q^{(u+Mr)(v+Ms)+3}\big(1-q^2\big)^{-2}\big(q^{u+Mr}+(-1)^{r'}q^{r'}\big)x^{u+Mr}y^{v+1+Ms} ,\\
\pi(\mathsf{f})(x)= c_0\big({-}q^{s'}-q^{-v-Ms}\big)x^{-u+1-Mr}y^{-v-Ms}\\
\hphantom{\pi(\mathsf{f})(x)=}{} + c_1\big({-}q^{s'}-q^{-v+(-M+1)s}\big)x^{-u+1+(-M+1)r}y^{-v+(-M+1)s} ,\\
\pi(\mathsf{f})(y)= c_0\big((-1)^{r'+1}q^{-r'-u-Mr}-1\big)x^{-u-Mr}y^{-v+1-Ms}\\
\hphantom{\pi(\mathsf{f})(y)=}{} + c_1\big((-1)^{r'+1}q^{-r'-u+(-M+1)r}-1\big)x^{-u+(-M+1)r}y^{-v+1+(-M+1)s} .
\end{gather*}
Here $r,s,u,v,M\in\mathbb{Z}$, $D=rv-su=4$; $G=\operatorname{gcd}(r,s)=2$; $r=2r'$, $s=2s'$; $c_0,c_1\in\mathbb{C}$, $c_0\ne 0$. This corresponds to $L=0$.

\subsubsection*{$\boldsymbol{D_4G_2E_2F_1(a)}$}
\begin{gather*}
\pi(\mathsf{k})(x)= -q^{-s'}x ,\qquad\pi(\mathsf{k})(y)=(-1)^{r'+1}q^{r'}y ,\\
\pi(\mathsf{e})(x)= a_0\big(1+q^{-s'+v+Ms}\big)x^{u+1+Mr}y^{v+Ms}\\
\hphantom{\pi(\mathsf{e})(x)=}{}+ a_1\big(1+q^{-s'+v+(M+1)s}\big)x^{u+1+(M+1)r}y^{v+(M+1)s} ,\\
\pi(\mathsf{e})(y)= a_0\big(q^{u+Mr}+(-1)^{r'}q^{r'}\big)x^{u+Mr}y^{v+1+Ms}\\
\hphantom{\pi(\mathsf{e})(y)=}{}+ a_1\big(q^{u+(M+1)r}+(-1)^{r'}q^{r'}\big)x^{u+(M+1)r}y^{v+1+(M+1)s} ,\\
\pi(\mathsf{f})(x)= a_0^{-1}q^{(u+Mr)(v+Ms)+3}\big(1-q^2\big)^{-2}\big({-}q^{s'}-q^{-v-Ms}\big)x^{-u+1-Mr}y^{-v-Ms} ,\\
\pi(\mathsf{f})(y)= a_0^{-1}q^{(u+Mr)(v+Ms)+3}\big(1-q^2\big)^{-2}\big((-1)^{r'+1}q^{-r'-u-Mr}-1\big)x^{-u-Mr}y^{-v+1-Ms} .
\end{gather*}
Here $r,s,u,v,M\in\mathbb{Z}$, $D=rv-su=4$; $G=\operatorname{gcd}(r,s)=2$; $r=2r'$, $s=2s'$; $a_0,a_1\in\mathbb{C}$, $a_0\ne 0$. This corresponds to $L=1$.

\subsubsection*{$\boldsymbol{D_4G_2E_1F_2(b)}$}
\begin{gather*}
\pi(\mathsf{k})(x)= (-1)^{s'+1}q^{-s'}x ,\qquad\pi(\mathsf{k})(y)=-q^{r'}y ,\\
\pi(\mathsf{e})(x)= -c_0^{-1}q^{(u+Mr)(v+Ms)+3}\big(1-q^2\big)^{-2}\big(1+(-1)^{s'}q^{-s'+v+Ms}\big)x^{u+1+Mr}y^{v+Ms} ,\\
\pi(\mathsf{e})(y)= -c_0^{-1}q^{(u+Mr)(v+Ms)+3}\big(1-q^2\big)^{-2}\big(q^{u+Mr}+q^{r'}\big)x^{u+Mr}y^{v+1+Ms} ,\\
\pi(\mathsf{f})(x)= c_0\big((-1)^{s'+1}q^{s'}-q^{-v-Ms}\big)x^{-u+1-Mr}y^{-v-Ms}\\
\hphantom{\pi(\mathsf{f})(x)=}{} + c_1\big((-1)^{s'+1}q^{s'}-q^{-v+(-M+1)s}\big)x^{-u+1+(-M+1)r}y^{-v+(-M+1)s} ,\\
\pi(\mathsf{f})(y)= c_0\big({-}q^{-r'-u-Mr}-1\big)x^{-u-Mr}y^{-v+1-Ms}\\
\hphantom{\pi(\mathsf{f})(y)=}{} + c_1\big({-}q^{-r'-u+(-M+1)r}-1\big)x^{-u+(-M+1)r}y^{-v+1+(-M+1)s} .
\end{gather*}
Here $r,s,u,v,M\in\mathbb{Z}$, $D=rv-su=4$; $G=\operatorname{gcd}(r,s)=2$; $r=2r'$, $s=2s'$; $c_0,c_1\in\mathbb{C}$, $c_0\ne 0$. This corresponds to $L=0$.

\subsubsection*{$\boldsymbol{D_4G_2E_2F_1(b)}$}
\begin{gather*}
\pi(\mathsf{k})(x)= (-1)^{s'+1}q^{-s'}x ,\qquad\pi(\mathsf{k})(y)=-q^{r'}y ,\\
\pi(\mathsf{e})(x)= a_0\big(1+(-1)^{s'}q^{-s'+v+Ms}\big)x^{u+1+Mr}y^{v+Ms}\\
\hphantom{\pi(\mathsf{e})(x)=}{} + a_1\big(1+(-1)^{s'}q^{-s'+v+(M+1)s}\big)x^{u+1+(M+1)r}y^{v+(M+1)s} ,\\
\pi(\mathsf{e})(y)= a_0\big(q^{u+Mr}+q^{r'}\big)x^{u+Mr}y^{v+1+Ms}\\
\hphantom{\pi(\mathsf{e})(y)=}{}+ a_1\big(q^{u+(M+1)r}+q^{r'}\big)x^{u+(M+1)r}y^{v+1+(M+1)s} ,\\
 \pi(\mathsf{f})(x)= a_0^{-1}q^{(u+Mr)(v+Ms)+3}\big(1-q^2\big)^{-2}\big((-1)^{s'+1}q^{-s'}-q^{-v-Ms}\big)x^{-u+1-Mr}y^{-v-Ms} ,\\
\pi(\mathsf{f})(y)= a_0^{-1}q^{(u+Mr)(v+Ms)+3}\big(1-q^2\big)^{-2}\big({-}q^{r'-u-Mr}-1\big)x^{-u-Mr}y^{-v+1-Ms} .
\end{gather*}
Here $r,s,u,v,M\in\mathbb{Z}$, $D=rv-su=4$; $G=\operatorname{gcd}(r,s)=2$; $r=2r'$, $s=2s'$; $a_0,a_1\in\mathbb{C}$, $a_0\ne 0$. This corresponds to $L=1$.

We conclude our list of symmetries with the following

\begin{proof}[\bf Proof of Main Theorem] The completeness of list of symmetries with $\sigma\ne I$ and the list of generic symmetries (those with the weight constants being subject to the assumptions of Theorem~\ref{gener}) has been established in~\cite{S}.

As for the non-generic symmetries listed in this Section, these are determined by setting the action of generators of $U_q(\mathfrak{sl}_2)$ on the generators of $\mathbb{C}_q\big[x^{\pm 1},y^{\pm 1}\big]$. To see that such an action extends to a well-defined $U_q(\mathfrak{sl}_2)$-symmetry on $\mathbb{C}_q\big[x^{\pm 1},y^{\pm 1}\big]$, one needs only to verify that everything passes through the relations in $U_q(\mathfrak{sl}_2)$ and in $\mathbb{C}_q\big[x^{\pm 1},y^{\pm 1}\big]$. This is a matter of routine calculations.

To see that the list of non-generic symmetries is complete, one has to observe that our exposition first separates out all the admissible collections of parameters for such symmetries (Section~\ref{wpem}), and then exhaust these collections in writing down the associated symmetries in Section~\ref{clngs}.
\end{proof}

\subsection*{Acknowledgements}

The author would like to thank the anonymous referees for a large number of comments and suggestions that substantially improved the initial version of this paper.

\pdfbookmark[1]{References}{ref}
\LastPageEnding


\begin{thebibliography}{99}
\footnotesize\itemsep=0pt

\bibitem{abe}
Abe E., Hopf algebras, \textit{Cambridge Tracts in Mathematics}, Vol.~74,
 Cambridge University Press, Cambridge~-- New York, 1980.

\bibitem{AD}
Alev J., Dumas F., Rigidit\'{e} des plongements des quotients primitifs
 minimaux de {$U_q(\mathfrak{sl}(2))$} dans l'alg\`ebre quantique de
 {W}eyl--{H}ayashi, \href{https://doi.org/10.1017/S002776300000595X}{\textit{Nagoya Math.~J.}} \textbf{143} (1996), 119--146.

\bibitem{DHL}
Duplij S., Hong Y., Li F., {$U_q(\mathfrak{sl}(m+1))$}-module algebra
 structures on the coordinate algebra of a quantum vector space,
 \textit{J.~Lie Theory} \textbf{25} (2015), 327--361.

\bibitem{DS}
Duplij S., Sinel'shchikov S., Classification of {$U_q(\mathfrak{sl}_2)$}-module
 algebra structures on the quantum plane, \textit{J.~Math. Phys. Anal. Geom.}
 \textbf{6} (2010), 406--430, \href{http://arxiv.org/abs/0905.1719}{arXiv:0905.1719}.

\bibitem{C}
Kassel C., Quantum groups, \textit{Graduate Texts in Mathematics}, Vol.~155,
 \href{https://doi.org/10.1007/978-1-4612-0783-2}{Springer-Verlag}, New York, 1995.

\bibitem{KPS}
Kirkman E., Procesi C., Small L., A {$q$}-analog for the {V}irasoro algebra,
 \href{https://doi.org/10.1080/00927879408825052}{\textit{Comm. Algebra}} \textbf{22} (1994), 3755--3774.

\bibitem{PLCCN}
Park H.G., Lee J., Choi S.H., Chen X.Q., Nam K.-B., Automorphism groups of some
 algebras, \href{https://doi.org/10.1007/s11425-009-0007-9}{\textit{Sci. China Ser.~A}} \textbf{52} (2009), 323--328.

\bibitem{S}
Sinel'shchikov S., Generic symmetries of the {L}aurent extension of quantum
 plane, \href{https://doi.org/10.15407/mag11.04.333}{\textit{J.~Math. Phys. Anal. Geom.}} \textbf{11} (2015), 333--358,
 \href{http://arxiv.org/abs/1410.8074}{arXiv:1410.8074}.

\bibitem{sweedler}
Sweedler M.E., Hopf algebras, \textit{Mathematics Lecture Note Series}, W.~A.~Benjamin,
 Inc., New York, 1969.

\end{thebibliography}
\end{document}